\documentclass[final,letterpaper]{siamart220329} 

\usepackage{amsmath,amssymb,mathtools,mathrsfs}
\usepackage{epsfig,graphics}
\usepackage{comment}
\usepackage{url}
\usepackage{color}
\usepackage{tikz}
\usepackage{pgfplots}
\usepackage{subfig}
\usepackage{graphicx}

\usepackage{diagbox}

\usepackage{tabularx}
\usepackage{array,booktabs}
\usepackage{siunitx}
\usepackage{enumitem}

\usepackage{algorithm}
\usepackage{algorithmicx}

\newcommand{\be}{\begin{equation}}
\newcommand{\ee}{\end{equation}}
\newcommand{\ba}{\begin{aligned}}
\newcommand{\ea}{\end{aligned}}

\newcommand{\cC}{{\mathcal C}}
\newcommand{\bbR}{{\mathbb R}}

\def\ccm{Center for Computational Mathematics, Flatiron Institute, Simons Foundation,
  New York, New York 10010}

\def\nyu{Courant Institute of Mathematical Sciences,
  New York University, New York, New York 10012}

\def\papertitle{An accurate and efficient scheme for function extensions on smooth domains}

\title{\papertitle}
\author{Charles L.~Epstein%
  \thanks{\ccm\,
    ({\tt cepstein@flatironinstitute.org}).}
  \and
  Fredrik~Fryklund\thanks{\nyu\,
  ({\tt nf2235@nyu.edu}).}
  \and
  Shidong~Jiang\thanks{\ccm\,
    ({\tt sjiang@flatironinstitute.org}).}
}

\begin{document}

\maketitle

\begin{abstract}
  A new scheme is proposed to construct an $n$-times differentiable function extension of an $n$-times
  differentiable function defined on a smooth domain, $D,$ in $d$-dimensions. The extension scheme
  relies on an explicit formula consisting of a linear combination of $n+1$
  function values in $D,$ which extends the function along directions normal to the boundary.
  Smoothness tangent to the boundary is automatic. The performance of the scheme is illustrated by
  using function extension as a step in a numerical solver for the inhomogeneous Poisson equation
  on multiply connected domains with {\em complex geometry} in two and three dimensions. We show that
  the modest additional work needed to do function extension leads to considerably more accurate
  solutions of the partial differential equation.
\end{abstract} 



\pagestyle{myheadings}
\thispagestyle{plain}
\markboth{C.~Epstein, F.~Fryklund, and S.~Jiang}
{An accurate and efficient scheme for function extensions on smooth domains}
\section{Introduction}\label{intro}
In a classic paper by Robert Seeley~\cite{seeley1964pams}, a simple formula is given
for the $\cC^\infty$ extension of a $\cC^\infty$ function defined in a half space. To be more precise,
let $y\in \mathbb{R}^{d-1}$, $x\in\mathbb{R}^1$, $\bbR^d_+=\mathbb{R}^{d-1}\times \{x\geq 0\}$,
then $\cC^{\infty}(\bbR^d_+)$ consists of infinitely differentiable functions defined in $\bbR^{d-1}\times (0,\infty)$
whose derivatives have continuous limits as  $x\rightarrow 0^+$.
For $x<0$, define the extended function
\be\label{cinfinityext}
E_{\infty}[f](y,x) = \sum_{j=0}^\infty w_j f(y, -t_j x) \phi(-t_j x), 
\ee
where $\{t_j\}\subset (0,\infty)$ is an unbounded, strictly increasing sequence, and $\phi$ is a $\cC^\infty$
window function on $\mathbb{R}^1$ with $\phi(x)=1$ for $0\le x\le 1$,
and $\phi(x)=0$ for $x\ge 2$.
For any $x<0,$ the sum in~\cref{cinfinityext} is finite and therefore $E_{\infty}[f]$ has the same smoothness
as the original function $f.$  In order to ensure the smoothness of the extension across the boundary, one needs
to require that the values of $E_{\infty}[f]$ and $f$  and all of their derivatives  match at $x=0$. This leads
to the following infinite system of linear equations for $\{w_j\}$
\be
\sum_{j=0}^\infty w_j t_j^i=(-1)^i, \qquad i=0, 1, \cdots,
\ee
i.e., the weight vector $w$ is the solution to an infinite Vandermonde system when
the node vector $t=(t_0,t_1,\dots)$ is given. Using $t_j=2^j$, Seeley showed that
$\sum_{j=0}^\infty |w_j| |t_j|^i < \infty$ for all $i\ge 0$. Thus, the extension operator $E_{\infty}$ 
is actually a continuous linear operator from $C^\infty(\mathbb{R}^d_+)$ to $C^\infty(\mathbb{R}^d)$.

Here we examine the extension formula \cref{cinfinityext} from the perspective of numerical
computation. Function extensions have broad applications in numerical analysis and scientific
computing, especially to boundary value problems for partial differential equations where
the boundary does not have simple geometry. Most existing numerical schemes rely on a data
fitting procedure using a choice of basis functions such as Fourier series, polynomials, radial
basis functions, etc. (see, for example,
\cite{askham2016cims,bruno2022sisc,demanet2019fcm,fryklund2022fmm,fryklund2018jcp,huybrechs2010sinum,STEIN2017155,webb2012sisc}).
The data fitting procedure is usually carried out on a volume grid to ensure the
smoothness of the extended function in all variables, which is expensive for problems in two and
higher dimensions. In~\cite{bruno2022sisc} a method is introduced that extends along the normal
direction using Fourier continuation.  There is also the \emph{active penalty method} \cite{activepenalty},
where an extension is created by using a set of basis functions that give global regularity $C^{k}$, once boundary
data and derivatives of order $k$ have been matched at the boundary. Observe that all of these schemes depend on
the original data linearly, regardless of the basis functions chosen to fit the data; that is, the extension
operator is a linear operator. 

As opposed to function extension, there is the method of \emph{function intension} \cite{STEIN2022111594}. The source
density is represented on a uniform Cartesian grid in the interior of the given computational domain, combined
with a conforming mesh in a tubular neighborhood of the boundary. The Poisson equation is solved in each region,
followed by enforcing consistency of the solution, across the interface of the interior of the domain and the
tubular neighborhood, by solving a boundary integral equation. Another method that relies on volumetric domain
decomposition is \cite{ANDERSON2023111688}, which is compatible with standard meshing software.

For applications to numerical analysis it is more appropriate to consider the finitely
differentiable analogue of $E_{\infty},$ given by
\be\label{cnext}
E_{n}[f](y,x) = \sum_{j=0}^n w_j f(y, -t_j x) \phi(-t_j x), 
\ee
where now we have the finite Vandermonde system:
\be
\sum_{j=0}^n w_j t_j^i=(-1)^i, \qquad i=0, 1, \cdots,n.
\ee
A nice feature of the extension formul{\ae} \cref{cinfinityext} and~\cref{cnext} is that the
extension acts along one direction only, while the smoothness in the other directions is
automatically guaranteed by construction. This feature can be extended to domains
$D\subset\bbR^d,$ with $\partial D$ a smooth embedded hypersurface. For $y\in\partial D,$
let $\nu_y$ denote the outer unit normal vector to $\partial D$ at $y.$ The tubular neighborhood
theorem, see~\cite{BottTu}, ensures that for some $\epsilon>0$ the map from
$\partial D\times (-\epsilon,\epsilon)$ given by
\be
(y,x)\mapsto y+x \nu_y
\ee
is a diffeomorphism onto its image.
For such a domain $D\subset\mathbb{R}^d$, we can extend the definition of $E_n$ by setting
\be\label{extformula}
E_n[f](y + x \nu_y)=\sum_{j=0}^n w_j f(y -t_j x \nu_y),
\ee
where $y$ is a point on $\partial D$. 
The window function is dropped for now, but will be needed in some applications. 

Once again, $E_n[f]$ has the same order of smoothness as $f$ in $\overline{D}^c$
and along the tangential directions on $\partial D$ by construction. Thus, one only needs to ensure
the continuity of the function and its normal derivatives up to order $n$
on $\partial D$ in order for the extension to be in $\cC^n$ in a small neighborhood of $D$.
This leads to the following $(n+1)\times (n+1)$ Vandermonde linear system for the nodes and weights
\be\label{vsystem}
\begin{bmatrix} 1&1&\cdots&1\\
  t_0&t_1&\cdots&t_n\\
  \vdots&\vdots&\cdots&\vdots\\
  t_0^n&t_2^n&\cdots&t_n^n
\end{bmatrix}\begin{bmatrix} w_0\\w_1\\\vdots\\w_n\end{bmatrix}
  =\begin{bmatrix} 1\\-1\\\vdots\\(-1)^n\end{bmatrix},
\ee
which we write as $Aw=c$ for short.  In general such systems are very poorly conditioned,
but, as we show in \cref{sec2}, the solution to this equation can be found explicitly,
without recourse to numerical linear algebra. 

Another very nice feature of this approach is that the nodes and weights are independent of the
domain, $D.$ By the nature of the extension operator, we need to have
$t_j\ge 0$ for $j=0, \ldots, n$ and without loss of generality, we assume that they are arranged
in increasing order, i.e., $0\le t_0<t_1<\cdots<t_n$. The extension $E_n[f]$ is then obtained by
simply evaluating the sums in~\cref{extformula}. While this does not require the solution of any
equations, it generally requires $f$ to be interpolated to the values of the arguments
$\{y-t_j x\nu_y:\: j=0,\dots,n\}$ that appear in the sum.

The paper is organized as follows. In \cref{sec2}, we present 
the main theoretical result. \Cref{sec3} discusses certain 
numerical issues associated with the application of our
function extension scheme to solving elliptic partial 
differential equations. Numerical results in one, two, and 
three dimensions are shown in \cref{sec4}.

\section{Main theoretical result}\label{sec2}
The main theoretical result of this paper is \cref{thm2},
which gives the optimal nodes and weights of the function extension formula \cref{extformula}.

\subsection{Solution to the Vandermonde system}\label{sec2.1}
It is well known that Vandermonde matrices are ill-conditioned (as is familiar in the problem of
function extrapolation). In \cite{gautschi1987nm}, it is shown that the condition number
(in the maximum norm) of the Vandermonde matrix is bounded from below by $(n+1)2^{n}$ when all
nodes are non-negative.  If the weights are calculated by solving \cref{vsystem} numerically,
then the condition number of the Vandermonde matrix will affect the accuracy of the weights.
However, \cref{vsystem} can be solved analytically by an elementary method, which we now present.
Denote by $B$ the inverse of the matrix  $A$. We then have
\be\label{vinverse}
\sum_{j=0}^n b_{ij}t_k^j=\delta_{ik}.
\ee
Consider the polynomial $p_i(x)=\sum_{j=0}^n b_{ij}x^j$. Then \cref{vinverse} is equivalent to
\be\label{lagrange}
p_i(t_k)=\delta_{ik}.
\ee
Recall that the  Lagrange basis functions $\{l_i(x)\}$ for interpolation through the nodes
$\{t_k, k=0,\cdots,n\}$  are given by
\be\label{lagrange2}
l_i(x)=\prod_{m=0, m\ne i}^n \frac{x-t_m}{t_i-t_m},
\ee
and the interpolant is then
\be\label{laginterp}
Q(x)=\sum_{i=0}^n f_i l_i(x).
\ee
It is straightforward to check that $l_i$ satisfies the condition \cref{lagrange}.
Thus,
\be
p_i(x)=l_i(x),
\ee
and
\be\label{weightformula}
\ba
w_i&=\sum_{j=0}^n b_{ij}(-1)^j=p_i(-1)=l_i(-1)\\
&=(-1)^i
\prod_{m=0}^{i-1} \frac{1+t_m}{t_i-t_m} \cdot \prod_{m=i+1}^{n} \frac{1+t_m}{t_m-t_i}.
\ea
\ee
That is, $w_i$ is the value of the $i$th Lagrange basis function evaluated at $-1$.

\subsection{Choice of optimal interpolation nodes}
To fix our discussion, let us now
assume that the $\cC^n$ function, $f,$ is defined in the interval $[0,a],$ and we extend $f$
to $[-1,0],$ which means that the nodes $\{t_i\}\subset [0,a]$ as well.
$E_n$ extends functions in $\cC^n([0,a])$ to functions in $\cC^n([-1,a]);$ we see that
\begin{equation}
    \|E_n(f)\|_{L^{\infty}([-1,0])}\leq \|w\|_1 \|f\|_{L^{\infty}([0,a])},
\end{equation}
where \be\label{onenorm} \|w\|_1=\sum_{i=0}^n |w_i|.  \ee Hence a good measure
of the ``quality" of the extension operator given by \cref{extformula} is the
$\ell^1$-norm of the weights $\{w_j\}.$ This norm, $\|w\|_1,$ is determined by
the interpolation nodes $\{t_i\}$ explicitly via the formula
\cref{weightformula}. We would therefore like to choose the interpolation nodes
$t_i\in [0,a]$ such that $\|w\|_1$ is minimized. While we are not explicitly
controlling the norms of the derivatives $\|\partial_x^l
E_n[f]\|_{L^{\infty}}([-1,0]),$ for $l=1,\dots,n,$ below we see that, for the
nodes and weights that minimize $\|w\|_1,$ these norms behave like
$a^l\|w\|_1\|\partial_x^lf\|_{L^{\infty}([0,a])}.$ In applications $a\approx 2, n\leq 10,$
so this is a very moderate rate of growth.

Combining \cref{laginterp} and \cref{weightformula}, we obtain
\be\label{onenorm2}
\ba
\|w\|_1&=\sum_{i=0}^n |w_i|\\
&=\sum_{i=0}^n(-1)^il_i(-1)\\
&=P(-1),
\ea
\ee
where $P(x)=\sum_{i=0}^n(-1)^il_i(x)$.
\begin{definition}
  A polynomial $P$ is said to satisfy the {\it equi-oscillation property of modulus $1$ on $[0,a]$}
  if there exist $t_i$ with $0\le t_0<t_1<\cdots<t_n\le a$
  such that $P(t_i)=(-1)^i$ for $i=0,\ldots,n$. 
\end{definition}

\begin{lemma}\label{lem1}
  If $P(x)$ is a polynomial of degree $n,$ satisfying the equi-oscillation property of modulus $1$,
  then its $n$ roots $x_i$ ($i=1,\cdots,n$) satisfy
  the property $t_0<x_1<t_1<x_2<t_2<\cdots<x_n<t_n$.
\end{lemma}
\begin{proof}
  This follows from the intermediate value theorem and
  the fact that  the signs of the values $\{P(t_0), P(t_1),\dots, P(t_n)\}$ alternate.
\end{proof}

\begin{lemma}\label{lem2}
  If $P(x)$ is a polynomial of degree $n,$ satisfying the equi-oscillation property of modulus $1$,
  then $P(x)>0$ decreases monotonically for $x<t_0$.
\end{lemma}
\begin{proof}
  This follows from \cref{lem1} and the fact that $P(t_0)=1$. That is,
  $P(x)=c\prod_{i=1}^n (x_i-x)$ for some $c>0$.
\end{proof}

It is easy to see that the choice of optimal interpolation nodes
for the purpose of function extension is equivalent to the following problem:\\
{\it Find the minimal value of $P(-1)$ among all polynomials of degree at most $n$ that
  satisfy the equi-oscillation property of modulus $1$ on $[0,a]$.}

Denote the minimal value by $m_a$, the corresponding polynomial by $p^\ast(x)$, 
and the associated nodes by $\{t_j^\ast\}$, respectively. In the following lemmas we give the
properties of the optimal nodes and $p^\ast$ itself.

\begin{lemma}\label{lem3}
  $t_0^\ast=0$.
\end{lemma}
\begin{proof}
  Suppose $t_0^\ast>0$. Then $q(x)=p^\ast(x+t_0^\ast)$ satisfies the alternating property
  $q(t_i^\ast-t_0^\ast)=p^\ast(t_i^\ast)=(-1)^i$. Furthermore, $q(-1)=p^\ast(-1+t_0^\ast)<p^\ast(-1)$
  by \cref{lem2}, which leads to a contradiction.
\end{proof}

\begin{lemma}\label{lem4}
  If $n>0$, then $t_n^\ast=a$.
\end{lemma}
\begin{proof}
  Suppose $t_n^\ast<a$. Consider the polynomial $q(x)$, which is obtained by replacing
  $t_n^\ast$ by $a$ and keep all other nodes unchanged. We claim that $q(-1)<p^\ast(-1)$.
  This can be seen from \cref{weightformula}, \cref{onenorm2}, and the facts
  that both $\frac{1}{t_n-t_m}$ and $\frac{1+t_n}{t_n-t_i}=1+\frac{1+t_i}{t_n-t_i}$
  decrease as $t_n$ increases. Hence the contradiction.
\end{proof}

\begin{lemma}\label{lem5}
If $a_1>a_2$, then $m_{a_1}<m_{a_2}$.
\end{lemma}
\begin{proof}
  By the definition of the problem, we have $m_{a_1}\le m_{a_2}$ since the space
  of allowable functions becomes  larger as $a$ increases. The strict inequality is
  achieved by \cref{lem4}.
\end{proof}

\begin{lemma}\label{lem6}
$|p^\ast(x)|\le 1$ for $x\in (0,a)$.
\end{lemma}
\begin{proof}
  Suppose first that $p^\ast(t^\circ)>1$ at some point $t^\circ\in (0,a)$.
  Let $t_j^\ast$ be the closest
  node to $t^\circ$ such that $p^\ast(t^\ast_j)=1$. It is clear that $j$ has to be even and
  $t^\circ<t_{j+1}^\ast$ by the
  alternating property. Consider the polynomial
  $q(x)$ which is obtained by replacing $t_j^\ast$ with $t^\circ$. Then
  $p^\ast(t_i^\ast)-q(t^\ast_i)=0$ for
  $i\ne j$. That is,
  \be
  p^\ast(x)-q(x)=c\prod_{i=0}^{j-1}(x-t^\ast_i)\cdot \prod_{i=j+1}^{n}(t^\ast_i-x),
  \ee
  Furthermore, since $p^\ast(t^\circ)-q(t^\circ)=p^\ast(t^\circ)-1>0$, we must
  have $c>0$. It then follows that
  \be\label{pgtq}
  \ba
  p^\ast(-1)-q(-1)&=c\prod_{i=0}^{j-1}(-1-t^\ast_i)\cdot \prod_{i=j+1}^{n}(t^\ast_i+1)\\
  &=c\prod_{i=0}^{j-1}(1+t^\ast_i)\cdot \prod_{i=j+1}^{n}(t^\ast_i+1)>0
  \ea
  \ee
  due to the facts that $c>0$ and $j$ is even. Hence the contradiction.

  Suppose that $p^\ast(t^\circ)<-1$. Then a similar argument leads to \cref{pgtq} again by the
  facts that $c<0$ and $j$ is odd.
\end{proof}

\begin{lemma}\label{lem7}
$|p^\ast(x)|< 1$ for $x\in (0,a)$ and $x\ne t_i$, $i=0,1,\ldots,n$. 
\end{lemma}
\begin{proof}
  This follows from \cref{lem6}, since otherwise there must be a point at which
  $|p^\ast|$ is greater than $1$.
\end{proof}

\begin{lemma}\label{lem8}
$p^{\ast\prime}(t^\ast_i)=0$ for $i=1,\ldots,n-1$. 
\end{lemma}
\begin{proof}
  By \cref{lem6} and \cref{lem7}, $p^\ast(x)$ achieves its interior extreme
  values at $t^\ast_i$. Hence, its first derivative vanishes at those interior
  nodes.
\end{proof}

The following lemma is key to the uniqueness.

\begin{lemma}\label{lem9}
  If $p_1(x)$ and $p_2(x)$ satisfy the equi-oscillation property of modulus $1$
  on $[0,a]$ and $|p_1(x)|\le 1$, $|p_2(x)|\le 1$, then $p_1(x)\equiv p_2(x)$.
\end{lemma}
\begin{proof}
  Suppose that $p_1$ is not identically equal to $p_2$. Then there exists a point $\xi<0$ such that
  $p_1(\xi)\ne p_2(\xi)>0$. Without loss of generality, assume that
  $0< p_1(\xi) < p_2(\xi)$. Consider the polynomial
  \be
  q(x)=p_1(x)-\frac{p_1(\xi)}{p_2(\xi)}p_2(x).
  \ee

  By construction, $q(\xi)=0$. Furthermore,
  \be
  \left|p_2(x)\frac{p_1(\xi)}{p_2(\xi)}\right|<1, \quad x\in [0,a],
  \ee
  since $p_1(\xi)<p_2(\xi)$ and $|p_2(x)|\le 1$ for $x\in [0,a]$.
  Thus, $q$ has the same signs as $p_1$ at its equi-oscillation nodes.
  By the intermediate value theorem, $q$ has $n$ zeros on $(0,a)$.
  Together with the zero at $\xi$, $q$ has $n+1$ zeros. Since $q$ is a polynomial
  of degree at most $n$, $q$ has to be identically equal to zero everywhere, which leads
  to a contradiction.
\end{proof}

Using these results we can now give an explicit formula for $p^\ast$ and $\{t_j^\ast\}:$
\begin{theorem}\label{thm1}
  $p^\ast$ is unique and $p^\ast(x)=T_n(1-2x/a)$, where $T_n$ is the Chebyshev polynomial
  of degree $n$.
\end{theorem}
\begin{proof}
  The uniqueness follows from \cref{lem9} and the existence follows from
  explicit construction. Namely, it is straightforward to check that
  $T_n(1-2x/a)$ satisfies the equi-oscillation property of modulus $1$ and $|T_n(1-2x/a)|\le 1$ 
  on $[0,a]$.
\end{proof}

Since we only used the fact that $-1$ is outside the interval $[0,a]$ in our proofs
and $[0,a]$ can be translated to any interval $[a,b]$, we actually showed the following
statement is true.
\begin{corollary}
  Suppose that $p(x)$ is a polynomial of degree at most $n$ satisfying the
  property $p(t_i)=(-1)^i$ or $p(t_i)=(-1)^{i+1}$ for $a\le t_0<t_1<\ldots<t_n\le b$.
  Then
  \be
  |p(x)|\ge \left|T_n\left(\frac{2}{b-a}\left(x-\frac{a+b}{2}\right)\right)\right|,
  \qquad x\notin [a,b].
  \ee
\end{corollary}

In other words, the translated and rescaled Chebyshev polynomials have the {\it least growth}
outside $[a,b]$ among all polynomials with the equi-oscillation property of modulus $1$ on $[a,b]$.
This should be compared to the fact, which appears as Theorem 1.10 in \cite{Rivlin}, that among
polynomials $p$ of degree $n,$ with sup-norm 1 on $[-1,1],$ the Chebyshev polynomial $T_n$ has the
following extremal properties: for any $x_0\notin(-1,1),$ and $0\leq k\leq n,$
\begin{equation}
    |p^{[k]}(x_0)|\leq |T^{[k]}_n(x_0)|.
\end{equation}

Using the explicit expressions of $T_n$
\be
T_n(x)=\left\{\begin{aligned}
&\cos(n\arccos(x)), \quad |x|\le 1 \\
&\frac{1}{2}\left(
\left(x-\sqrt{x^2-1}\right)^n + \left(x+\sqrt{x^2-1}\right)^n
\right), \quad |x|\ge 1, \end{aligned}\right.
\ee
We may calculate the optimal weights explicitly. Recall that $T_n(\cos\frac{i\pi}{n})=(-1)^i;$
consider the function
\be
\phi(x)=\prod_{i=0}^n (x-t_i).
\ee
Then
\be
l_i(x)=\frac{\phi(x)}{(x-t_i)\phi'(t_i)},
\ee
and
\be
w^\ast_i=l_i^\ast(-1)=-\frac{\phi^\ast(-1)}{(1+t_i)\phi^{\ast\prime}(t_i)}.
\ee

We first work on the standard interval $[-1,1]$.
By \cref{lem8}, the $n-1$ interior nodes are the zeros of $T'_n(x)$.
Thus, $\phi(x)=c(x^2-1)T'_n(x)$ for the interval $[-1,1]$. Back to $[0,a]$, we have
\be
\phi^\ast(x)=x(x-a)T'_n(2x/a-1),
\ee
and
\be
\phi^{\ast\prime}(x)=\frac{2}{a}x(x-a)T''_n(2x/a-1)+(2x-a)T'_n(2x/a-1),
\ee
where the irrelevant constant $c$ is dropped.
\begin{theorem}\label{thm2}
  The optimal nodes on $[0,a]$ for function extension formula~\cref{extformula} are
  Chebyshev nodes of the second kind shifted and scaled to the interval $[0,a]$:
  \be\label{optnodes}
  t^\ast_i=\frac{a}{2}\left(1-\cos\left(\frac{i\pi}{n}\right)\right), \quad i=0,\ldots,n.
  \ee
  And the associated optimal weights are
  \be\label{optweights}
  w^\ast_i=(-1)^i\frac{C_n(a)}{(1+\delta_{i0}+\delta_{in})
    na(1+t_i^\ast)}, \quad i=0,\ldots,n,
  \ee
  where
  \be
  C_n(a)=\frac{1+a}{\sqrt{x_0^2-1}}\left(\left(x_0+\sqrt{x_0^2-1}\right)^n
  -\left(x_0-\sqrt{x_0^2-1}\right)^n\right), \quad x_0=1+\frac{2}{a}.
  \ee
  Finally, the $l_1$ norm of the associated optimal weights is
  \be
  \ba
  \|w^\ast\|_1&=\sum_{i=0}^n |w^\ast_i|=T_n(1+2/a)\\
  &=\frac{1}{2}\left(\left(1+\frac{2}{a}-2\sqrt{\frac{1}{a}+\frac{1}{a^2}}\right)^n
  +\left(1+\frac{2}{a}+2\sqrt{\frac{1}{a}+\frac{1}{a^2}}\right)^n\right).
  \ea\label{conditionnumber}
  \ee
\end{theorem}

The function extension is, as expected, exponentially
ill-conditioned. As $a$ increases, the condition number decreases, due to the fact that
$-1$ is relatively closer to the origin with respect to the interval $[0,a]$;  for $a>1$
the condition number is approximately $e^{\frac{2n(\sqrt{a}+1)}{a}}.$ Since the extension
formula~\cref{extformula} only imposes the minimal conditions to ensure $\cC^n$ continuity across
the boundary, it seems quite reasonable that \cref{conditionnumber} can be viewed as the
{\it intrinsic condition number} of any linear function extension scheme.

The formula for $E_n[f]$ can be differentiated to give
\begin{equation}
  \partial_x^l E_n[f](x)=\sum_{j=0}^n(-t_j)^lw_jf^{[l]}(-t_j x),
\end{equation}
and therefore
\begin{equation}
  |\partial_x^l E_n[f](x)|\leq \left[\sum_{j=0}^nt_j^l|w_j|\right]\|f^{[l]}\|_{L^{\infty}([0,a])},
  \text{ with }x\in[-1,0].
\end{equation}
For the optimal weights we see that
\begin{equation}\label{eqn34}
  \sum_{j=0}^n t^{* l}_j|w^*_j|\leq a^l\|w^*\|_1
  \frac{\sum_{j=0}^n\frac{\left(\frac 12(1-\cos\left(\frac{\pi j}{n}\right))\right)^l}
    {(1+t_j^*)(1+\delta_{j0}+\delta_{jn})}}
       {\sum_{j=0}^n\frac{1}{(1+t_j^*)(1+\delta_{j0}+\delta_{jn})}}
\end{equation}
The ratio of sums is certainly less than 1, but also bounded below by $1/2(1+a)(n+1).$  

As $n$ gets large, the numerator and denominator in~\cref{eqn34} can be seen as $n/\pi$ times
the trapezoidal approximations to the integrals
\begin{equation}
  \int_{0}^{\pi}\frac{\left[\frac 12(1-\cos x)\right]^l dx}{1+\frac a2(1-\cos x)}\text{ and }
  \int_{0}^{\pi}\frac{dx}{1+\frac a2(1-\cos x)},
\end{equation}
respectively. Using contour integration to evaluate these integrals one can show that the
denominator is approximated by $\frac{n}{\sqrt{1+a}}$ and the numerator by
\begin{equation}
  \frac{n}{a2^{2(l-2)}}\sum_{j=0}^{l-1}\left(\begin{matrix} 2(l-j-1)\\l-j-1\end{matrix}\right)
    \left(\frac{-4}{a}\right)^j+8(-1)^{l+1}\frac{n}{a^l\sqrt{1+a}}.
\end{equation}
Clearly, as $a$ grows, the $j=0$ term dominates the numerator, and therefore we have the estimate
\begin{equation}
  \sum_{j=0}^n t^{* l}_j|w^*_j| \leq C_l\frac{1}{2^{2l}}\left( \begin{matrix} 2(l-1)\\l-1\end{matrix}
    \right)a^{l-\frac 12},
\end{equation}
where $C_l$ is uniformly bounded as a function of $a,l,n.$

\begin{table}[t]
\caption{\sf Condition number of function extension formula~\cref{extformula} in $l_1$ norm.
  The first row lists the order of the function extension.
  The first column lists the size of the interval.}
\sisetup{tight-spacing=true,exponent-product = \cdot,table-format=3.1e-2}
\centering
\begin{tabular}{
    |S[table-format=1.3]|
    S[scientific-notation = true,round-mode=places,round-precision=1,table-format=3.1e-2]
    S[scientific-notation = true,round-mode=places,round-precision=1,table-format=3.1e-2]
    S[scientific-notation = true,round-mode=places,round-precision=1,table-format=3.1e-2]
    S[scientific-notation = true,round-mode=places,round-precision=1,table-format=3.1e-2]
    S[scientific-notation = true,round-mode=places,round-precision=1,table-format=3.1e-2]
    S[scientific-notation = true,round-mode=places,round-precision=1,table-format=3.1e-2]
    S[scientific-notation = true,round-mode=places,round-precision=1,table-format=3.1e-2]
    |}
\hline
{\diagbox{$a$}{$n$}} & 2 & 3 & 4 & 5 & 6 & 7 & 8\\
\hline
0.15 & 409.9 & 11735.8 & 336017 & 9620746 & 275458698 & 7886861923 & 225814583082\\
0.25 & 161.0 & 2889.0 & 51841 & 930249 & 16692641 & 299537289 & 5374978561\\
0.5 & 49.0 & 485.0 & 4801 & 47525 & 470449 & 4656965 & 46099201\\
1 & 17.0 & 99.0 &  577 & 3363 & 19601 & 114243 & 665857 \\
2 & 7.0 & 26.0 &   97 &  362 & 1351 & 5042 & 18817\\
4 & 3.5 &  9.0 &   24 &   62 &  161 &  422 &  1104\\
\hline
\end{tabular}
\label{table1}
\end{table}

\section{Numerical algorithms}\label{sec3}
\subsection{Window function}
\Cref{table1} lists condition numbers of the function extension formula~\cref{extformula},
where the first row lists the extension order $n$, and the first column lists the size
of the interval, $a,$ where $f$ is defined. We observe that the condition numbers do
increase exponentially fast as $n$ increases. For most practical applications,
function extensions are very often used as an intermediate step. This is true for both
FFT-based solvers of boundary value problems of partial differential equations and
evaluation of volume potentials in integral equation methods. In this case, 
the high condition number of the function extension step should have a fairly
mild effect on the overall accuracy of the method. This is illustrated by numerical examples
presented in \cref{2dexamples}.

When the FFT is used to treat the inhomogeneous term in the governing PDE in a domain
with complex boundary, the inhomogeneous term is extended to a rectangular box containing
the original domain. In order for the FFT-based solver to achieve spectral convergence
rate, the extended inhomogeneous term must be smoothly rolled off to zero. For this, we
add a window function in our extension formula, i.e.,
\be\label{fextfinal}
E_n[f](x)=\left(\sum_{i=0}^nw_if(-t_ix)\right)\Phi^c(-x,r_0,r_1),
\ee
where $\Phi^c$ is a smooth
window function that equals one for $|x|\le r_0,$ in order to maintain the
$\mathcal{C}^n$ smoothness of the extension at the origin,
and vanishes for $|x|>r_1;$  $r_1$ is a parameter
to be determined. In order to minimize the additional frequency content introduced
by the function extension, we use prolate spheroidal wave functions (PSWFs) of order zero
to construct the window function. For any positive real number $c$, the first
prolate spheroidal wave function of order zero, denoted as $\psi_0^c$,
is the eigenfunction of the integral operator
$F_c: L^2[-1,1] \rightarrow L^2[-1,1]$, defined by the formula
\be\label{pswfintegraloperator}
F_c[\phi](x) = \int_{-1}^1 \phi(t)e^{icxt}dt,
\ee
corresponding to the eigenvalue of the largest magnitude, see~\cite{SlepianPollackI}.
Though the PSWFs are defined on the whole real axis, only its values on the ``standard''
interval $[-1,1]$ are needed for many practical applications. The graph of $\psi_0^c$ looks
like a bell, and is similar to a Gaussian $e^{-x^2/\delta}.$  The advantage of $\psi_0^c$ over
Gaussians, when used as a ``window'' function, lies in the fact that the width of the
Fourier spectrum of $\psi_0^c$ is roughly half of that of the corresponding Gaussian,
when both $\psi_0^c$ and the Gaussian are normalized to have effective support in
$[-1,1]$ to the same precision $\epsilon$.  Mathematically, it has been proven
that $\psi_0^c$ is the optimal window function in $L^2$ norm. The value of the parameter
$c$ depends on the requested precision logarithmically. See \cref{pswfgaussian}.

\begin{figure}[!ht]
\centering 
\includegraphics[height=40mm]{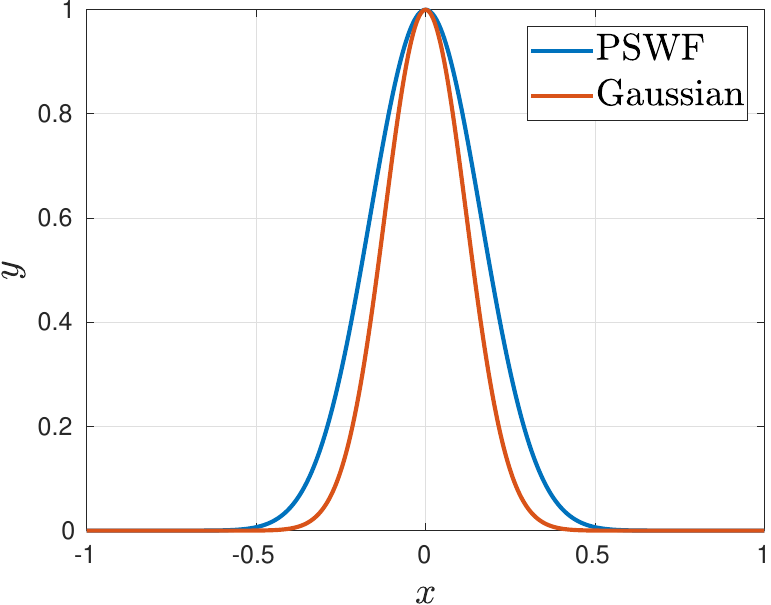}
\hspace{2mm}
\includegraphics[height=40mm]{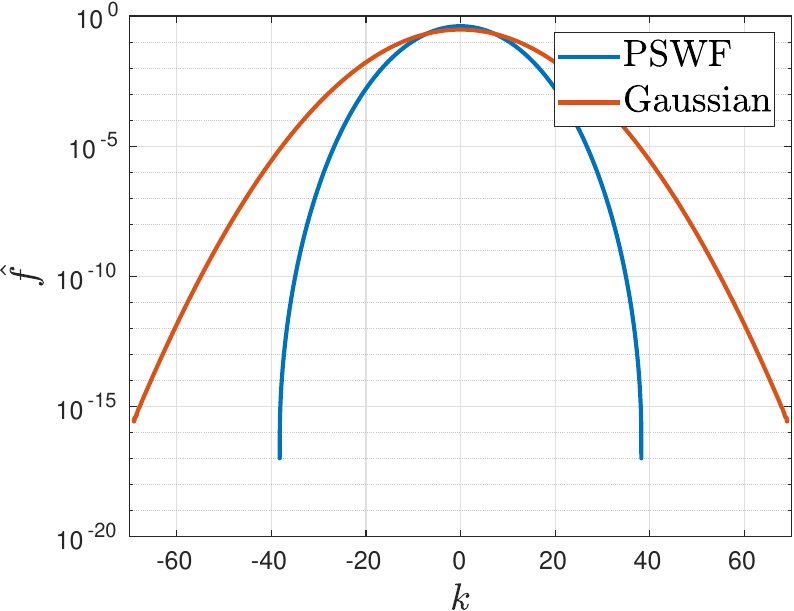}
\caption{\sf Comparison of $\psi_0^c(x)$ and the Gaussian. Left: $\psi_0^c(x)/\psi_0(0)$ with
  $c\approx 38.242000579833984$ and the Gaussian $G(x)=e^{-x^2\delta^2}$ with
  $\delta\approx \sqrt{\log(1/\epsilon)}$, where $\epsilon=10^{-15}$. $\psi_0^c(\pm 1) \approx
  G(\pm 1) \approx \epsilon$. Right: their Fourier transforms. Note that the Fourier transform
  of the Gaussian is almost twice as wide as that of the PSWF.
}
\label{pswfgaussian}
\end{figure}

Let
\be
\Psi_0^c(x)=
\begin{cases} &0, \quad x\le -1\\
& \frac{1}{C_0}\int_{-1}^x\psi_0^c(t)dt, \quad C_0=\int_{-1}^1\psi_0^c(t)dt, \quad x\in[-1,1]\\
& 1, \quad x\ge 1,\end{cases}
\ee
then define the window function in \cref{fextfinal} via the formula
\be\label{windowfunc}
\Phi^c(x,r_0,r_1)=1-\Psi_0^c\left(\frac{2x-(r_0+r_1)}{r_1-r_0}\right).
\ee
It is easy to verify that $\Phi^c(x)=1$ for $x\le r_0$ and $\Phi^c(x)=0$ for $x\ge r_1$.
Furthermore, even though $\psi_0^c$ is sharply truncated to zero and thus discontinuous at $\pm 1$,
$\Phi^c(x)$ can be viewed as a smooth function for numerical purpose to the requested precision.
Finally, $\Phi^c$ can be computed efficiently via the polynomial approximation of $\psi_0^c$.
At $15$ digits of accuracy, $\psi_0^c$ is well approximated by an even polynomial containing
$28$ nonzero terms. 

\subsection{Application to FFT-based elliptic PDE solvers}
In \cite{fryklund2018jcp}, an FFT-based solver has been outlined to solve
the Poisson equation on complex domains in two dimensions. The function extension
scheme in \cite{fryklund2018jcp} combines radial basis functions and partitions of unity
to achieve high-order function extension. With our current function extension scheme along
normal directions, it is straightforward to extend the numerical algorithm
in \cite{fryklund2018jcp} to solve boundary value problem of any constant-coefficient
elliptic partial differential equations on complex domains in both two and three dimensions,
e.g.,
\be
\ba
\mathcal{L} u(x) &= f(x), \quad x\in \Omega \subset \mathbb{R}^d,\\
u(x) &= g(x), \quad x\in \partial \Omega,
\ea\label{ellipticbvp}
\ee
where $\mathcal{L}$ is a constant-coefficient second-order elliptic differential operator and
$d=2, 3$ (other boundary conditions can be treated in a similar manner). 
The Green's function for the operator $\mathcal{L}$ is denoted by $G(x,y)$.
Here, we summarize the algorithmic steps for solving the elliptic BVPs on complex
domains in two and three dimensions.
We then discuss some implementation details of the scheme.

\subsubsection{Summary of the FFT-based elliptic PDE solvers on complex domains}
\label{schemeoutline}
For FFT-based solvers, the (possibly multiply connected) complex domain $\Omega$
is enclosed by a rectangular box. The rectangular box is discretized using an equi-spaced
grid. We assume that the boundary $\partial \Omega$ is sufficiently smooth and discretized
into a collection of patches. Each patch is further discretized into certain number of
collocation nodes such that a $p$th order integrator is available to integrate smooth functions.
The nodes on the boundary are denoted as $x^b_i$ for $i=1,\ldots, N_b$, where $N_b$ is
the total number of discretization points on the boundary.
We also assume that function handles are given for both $f$ and $g$ in \cref{ellipticbvp}
so that both functions can be evaluated anywhere efficiently in $\Omega$ and on $\partial \Omega$,
respectively. For simplicity, we assume that the step size is $h_0$ in all dimensions.
The FFT-based solver for \cref{ellipticbvp} is summarized as follows.

\begin{enumerate}[label=(\Roman*)]
\item Find the side length of the bounding box via the formula
  \be
  L = \max_{i=1}^{N_b} |x^b_i| + 48 h_0.
  \ee
  Set the number of points $M$ in each dimension to the even integer nearest to $2L/h_0$.
  Calculate the actual step size by $h=2L/M$ and fix $M^d$ equi-spaced points $x^{\rm vol}_i$
  on a tensor grid lying in $[-L,L)^d$. The cost of this step is $O(N_b)$.

  \item Separate the tensor grid $x^{\rm vol}_i$ into two classes - interior points $x^{\rm int}_i$
    ($i = 1, \ldots, N_{\rm int}$) and exterior points $x^{\rm ext}_i$
    ($i = 1, \ldots, N_{\rm ext}$). The cost of this step is $O(N_b+M^d)$.

  \item Extend the function $f$ to the whole computational box, i.e., calculate the values
    of the extended function $f^e$ at the exterior grid points $\{x^{\rm ext}_i\}$.
    The cost of this step is $O(N_b)$.

  \item Use the FFT to calculate a particular solution to $\mathcal{L}u^p = f^e$.
    The cost of this step is $O(M^d \log M)$

  \item Use NUFFT to calculate the values of the particular solution on the boundary. By
    the linearity of the problem and the superposition principle, the solution $u$ to
    \cref{ellipticbvp} is the sum of $u^p$ and $u^c$, where the correction $u^c$ is the solution
    to the problem
    \be
    \ba
    \mathcal{L} u^c(x) &= 0, \quad x\in \Omega \subset \mathbb{R}^d,\\
    u^c(x) &= g^c(x) := g(x)-u^p(x), \quad x\in \partial \Omega.
    \ea\label{purebvp}
    \ee
    That is, $u^c$ satisfies the homogeneous PDE in $\Omega$ with the boundary data equal to $g(x)-u^p(x)$.
    The cost of this step is $O(N_b+M^d\log M)$.
    
  \item Represent $u^c$ via suitable linear combination of layer potentials. For example,
    for the Dirichlet problem, $u^c(x) = (\mathcal{D}+c \mathcal{S})[\sigma](x)$, where
    $\mathcal{D}$ and $\mathcal{S}$ are the double and single layer potentials, respectively,
    and $c$ is a constant depending on the coefficients in the differential operator $\mathcal{L}$.
    This leads to a well-conditioned second-kind Fredholm boundary integral equation (BIE)
    for the unknown density $\sigma$. Discretize the resulting BIE via high-order quadrature
    (for example, the kernel-split quadrature in \cite{helsing2018jcp}), and solve it via GMRES and
    the fast multipole method (FMM). The cost of this step is $O(N_b)$.

  \item Use the FMM and high-order quadrature (see, for example, \cite{helsing2018jcp,zhu2022sisc})
    to evaluate $u^c$ on the tensor grid. The cost of this step is $O(N_b +M^d)$.

  \item Set $u=u^p+u^c$ on the tensor grid. The cost of this step is $O(M^d)$.
\end{enumerate}
The overall cost of this method is $O(N_b+M^d\log M).$

\subsubsection{Some implementation details}
In various steps (II, III, VII) of the scheme outlined above, we need to identify
target points that are close to the boundary, or to each patch on the boundary, wherein one needs to
apply a special quadrature to evaluate nearly singular integrals accurately. For FFT-based solvers,
all target points lie on an equi-spaced tensor grid. So they are already ordered and sorted. For each patch
on the boundary, one may find all target points that are at most $d_{\rm max}$ away from the patch, where
$d_{\rm max}$ is a user-specified parameter for function extension or close evaluation. This can be carried
out by finding an enclosing rectangular box whose boundary is $d_{\rm max}$ away from the patch, identifying
all target points inside that rectangular box, and removing target points that are too far from the patch.
This is essentially a {\em bin sort} algorithm. Obviously, the calculation is {\em local} for each
patch on the boundary and the cost of this algorithm is $O(N_b)$.

In $2d$ step (II) can be carried out as follows: 
By Gauss' lemma \cite{kress2014},
\be
\int_{\partial \Omega}\frac{\partial G_{\rm L}(x,y)}{\partial \nu(y)}ds(y)=
\left\{\begin{aligned} &-1, \quad x\in \Omega,\\
& -1/2, \quad x\in \partial \Omega,\\
& 0, \quad x\in \mathbb{R}^d\setminus \bar{\Omega},\end{aligned}\right.
\label{gausslemma}
\ee
where $G_{\rm L}$ is the standard Newtonian potential function in $\mathbb{R}^d$. Thus, one may determine whether
a given target point $x$ lies in the interior or exterior of $\Omega$ by evaluating the integral
in \cref{gausslemma}. When the domain is multiply connected, one may use a constant density that is
equal to index of the boundary to determine which domain the target point lies in.
The integral in \cref{gausslemma} can be discretized via smooth quadrature when the target
$x$ is far away from the source patch, and using  a high-order quadrature, as in
\cite{helsing2018jcp,zhu2022sisc}, when $x$ is close to the source patch. In order to use
fast algorithms such as the FMM, we write the high-order quadrature for close evaluation
as a correction of the smooth quadrature. The cost of the FMM for computing the action
of smooth quadrature is $O(N_b+M^d)$, where the prefactor depends on the precision
logarithmically. Here, due to the binary classification of points,
the accuracy of the FMM can be set very low ($10^{-2}$ in our implementation). The cost of
close evaluation correction is $O(N_b)$ since the number of targets close to each patch
is $O(1)$.

Step (III) can be carried out as follows: First, we identify target points that require
the function extension (some targets are far away from the boundary and the function values
are simply set to zero). For each such target point, we identify the closest source
point on the boundary. The cost of these two steps is $O(N_b),$ using the {\em bin sort} outlined
above. We then use Newton's method to find the point on the boundary whose normal passes
through the given target point, with the initial guess being the closest source point on the
boundary. For most practical cases, Newton's method converges rapidly with this initial guess.
Finally, we apply the normal extension formula to calculate the value of the extended
function at the given target.
 
In three dimensions we find that the approach by using  Gauss' lemma is slower than to use a
k-d data structure from \cite{kennel2004kdtree} for finding the closest discretization point
on the surface to a given a target point. A lookup operates at $\mathcal{O}(N_b)$ and generating
the k-d data structure operates at about eight million points per second per core. Once we have
the closest surface point, we proceed as above with Newton's method if the distance from the
target point to closest surface point is less or equal to $r_{1}$. Otherwise, we declare the
point given by lookup in the k-d data structure to be the closest point. To sort points as
inside or outside we look at the sign, relative to the outward normal, of the vector from the
target point to the declared closest point on the surface.

\section{Numerical results}\label{sec4}
We first show some illustrative function extension examples in one dimension, and then various
implementations of the algorithm described above in two and three dimensions.
We have implemented the numerical scheme for two dimensions in MATLAB with the FMM2D library from \cite{fmm2dlib}
and the NUFFT library from \cite{finufftlib}. The numerical results are obtained in MATLAB R2023a
on a machine with eight AMD Ryzen 7 PRO 6850U cores and 16MB cache. Since some parts
of the code are written in the single-thread mode, we use the Matlab command {\tt cputime}
to calculate the execution time for 2D examples.
The numerical scheme for three dimensions is implemented in Julia $1.8.0$ and Fortran. We use
FMM3DBIE \cite{fmm3dbielib} as the solver for the Laplace equation through a boundary integral
equation formulation \cite{GREENGARD2021100092}, the NUFFT library from \cite{finufftlib} and
the FMM3D library from~\cite{FMM3D}. 

\subsection{One dimension}
We first demonstrate the quality of the function extension formula~\cref{fextfinal} in one
dimension.
We use $D=[-1,1]$ as the interval where the original function is defined, and extend
the function to $[-L_h,L_h]$, where $L_h=1+32h$ with $h$ being the step size used in the
discrete Fourier transform. Since function extensions are used very often as an intermediate
step in numerical computation, as in the fast integral methods for solving elliptic PDEs
using a uniform grid, the choice above of $L_h$ is equivalent to having extra $32$ nodes outside
the original computational domain along the normal directions. Obviously, this would lead to the
ratio of number of interior nodes to that of ``artificial'' exterior nodes approaching
$0$ in the limit of $h \rightarrow 0$. One could choose the computational domain $[-L,L]$ such
that $L$ is independent of $h$, which would lead to a fixed ratio of number of interior nodes to
that of exterior nodes. To avoid caustics in dimensions greater than 1, the extension domain needs
to sufficiently small near points where the boundary is not convex. In our PDE applications,
using a fixed number of extra nodes seems to suffice.

In the following numerical experiments, we set the function extension order to $n=8$, $a=1$
in the formulae~\cref{optnodes}--\cref{optweights} for computing extension nodes and
weights, and the parameters in the window
function~\cref{windowfunc} to $c=40.590000152587891$ (i.e., 16 digits of accuracy),
$r_0=10^{-6}$, $r_1=32h$. 
We measure the quality of the function extension via the following quantities.
\begin{enumerate}[label=(\alph*)]
\item $\kappa=\|E_n[f]\|_\infty/\|f\|_\infty$, i.e., $\kappa$ is the ratio
  of the maximum norm of $E_n[f],$ defined in~\cref{fextfinal}, to that of the original function $f$.
  It is clear that the lower the value of $\kappa$, the better quality of $E_n[f]$. 
\item $\kappa_0=\|E_0[f]\|_\infty/\|f\|_\infty$, where $E_0[f]$ is the extended function without
  any window function. The difference between $\kappa$ and $\kappa_0$ shows the effect
  of the window function. In practice, both $\kappa$ and $\kappa_0$ depends on $h$. This is
  because smaller $h$ leads to the function extension in a smaller domain using data points
  closer to the boundary points. 
\item Relative $l^2$ error defined by the formula \be\label{error1d}
  E=\|\tilde{f}_{h/8}-f_{h/8}\|/\|f_{h/8}\|, \ee where $\|f_{h/8}\|$ is the
  $l^2$ norm of the original function $f$ on an eight-fold oversampled grid,
  $\tilde{f}_{h/8}$ is a column vector consisting of values of the inverse
  Fourier transform of $\hat{F}_h$ on the same oversampled grid, with
  $\hat{F}_h$ the Fourier transform of the extended function, $E_n(f),$ on
  $[-L_h,L_h]$ evaluated on the original grid. That is, $E$ measures the quality
  of the function extension by its Fourier transform.
\end{enumerate}
We use the following functions in our numerical experiments.
\be\label{testfunctions}
\ba
f_1(x)&=J_0(35(x+0.2)),\\
f_2(x)&=x^2e^{-30(1-x^2)},\\
f_3(x)&=T_{22}(x)=\cos(22\arccos(x)),
\ea
\ee
where $J_0$ is the first kind Bessel function of order zero, and $T_{22}$
is the Chebyshev polynomial of degree $22$. Note that $f_1$ is an oscillatory
function, $f_2$ has a boundary layer on the endpoints $\pm 1$, and $f_3$ is
the function that attains the largest condition number in our derivation.

\Cref{fig1}--\cref{fig3} present numerical results for these three
functions. Here, the left panel shows the extended function $E_0[f]$ without
the window function on $[-L_h,L_h]$ with $h=0.0037315$ (in red) and the original function
on $[-1,1]$. The middle panel shows the extended function $E_n[f]$ with the window
function. The right panel shows the convergence rate of the extension scheme, i.e.,
relative $l^2$ error $E$ defined in \cref{error1d} as a function of the step size $h$.
The parameters in these three functions are chosen such that the relative $l^2$ error
for the largest $h$ is $O(1)$. We use 25 equi-spaced step sizes in log10 scale on $[-1.22,-3.45]$,
with the largest step size chosen such that $L_{h_{\rm max}}<3$, that is, the function extension
only needs the function values on the original interval $[-1,1]$. 

\begin{figure}[!ht]
\centering 
\includegraphics[height=33mm]{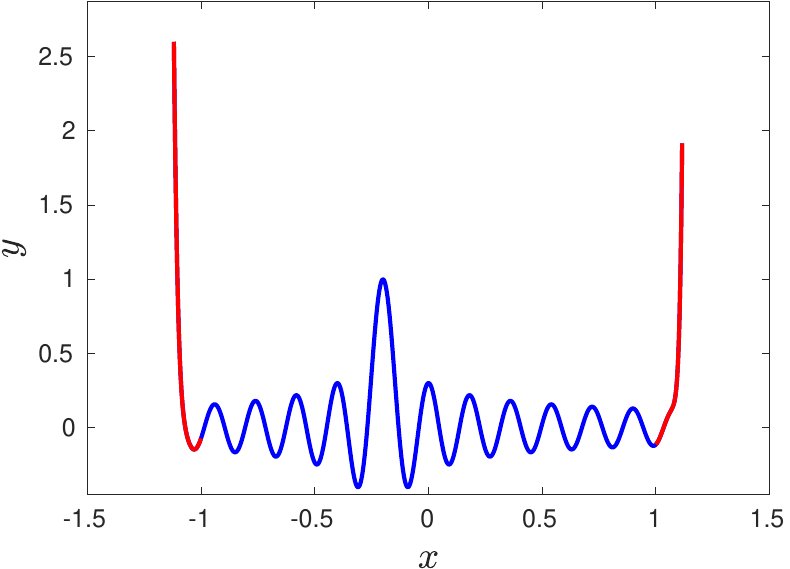}
\hspace{2mm}
\includegraphics[height=33mm]{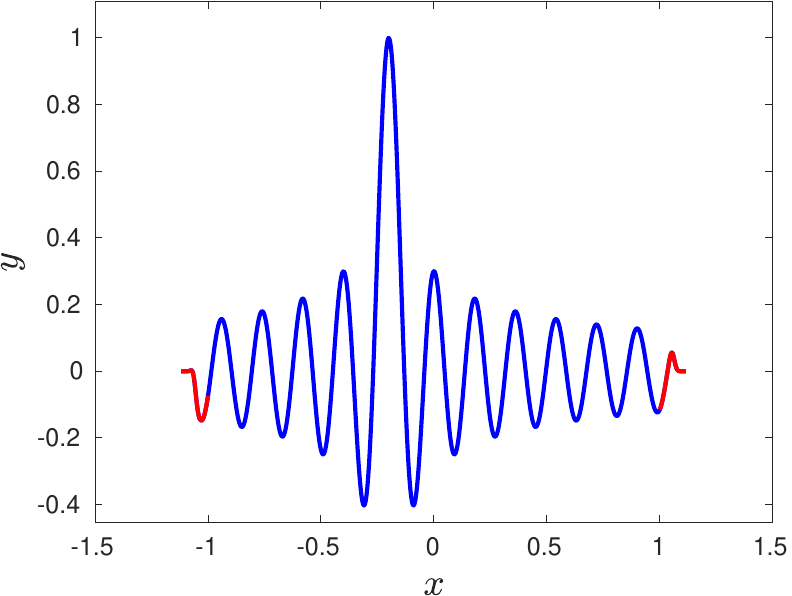}
\hspace{2mm}
\includegraphics[height=33mm]{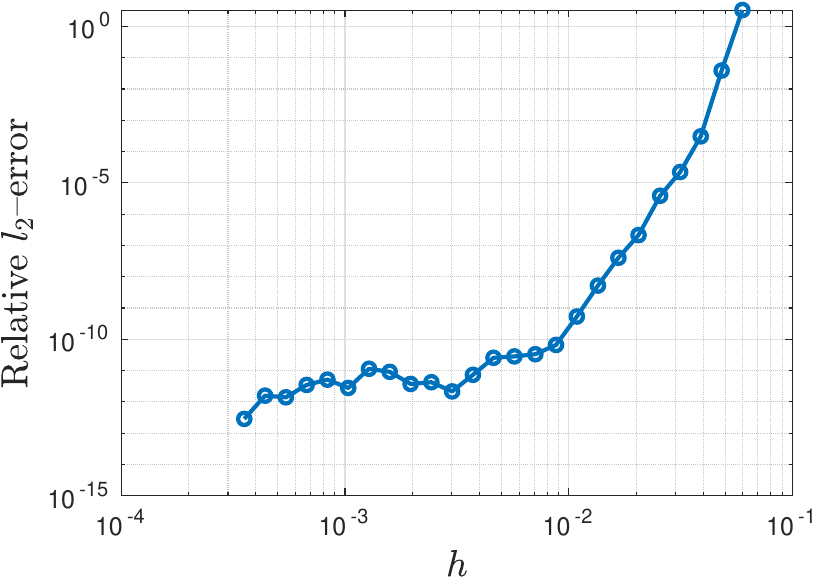}
\caption{\sf Numerical results for extending the oscillatory function $f_1(x)=J_0(35(x+0.2))$
  from $[-1,1]$ to roughly $[-1.12,1.12]$. The extension order is $n=8$ and the extension
  parameter $a$ is set to $1$. Left: the original function (in blue) and the extended
  function without any window function (in red),
  $\|E_0[f]\|_\infty/\|f\|_\infty\approx 2.6$.
  Middle: the original function (in blue) and the extended
  function with a window function (in red),
  $\|E_n[f]\|_\infty/\|f\|_\infty= 1$.
  Right: convergence study of the extension scheme, using the relative $l^2$-error, $E,$
  defined in~\cref{error1d}.
}
\label{fig1}
\end{figure}

\begin{figure}[!ht]
\centering 
\includegraphics[height=33mm]{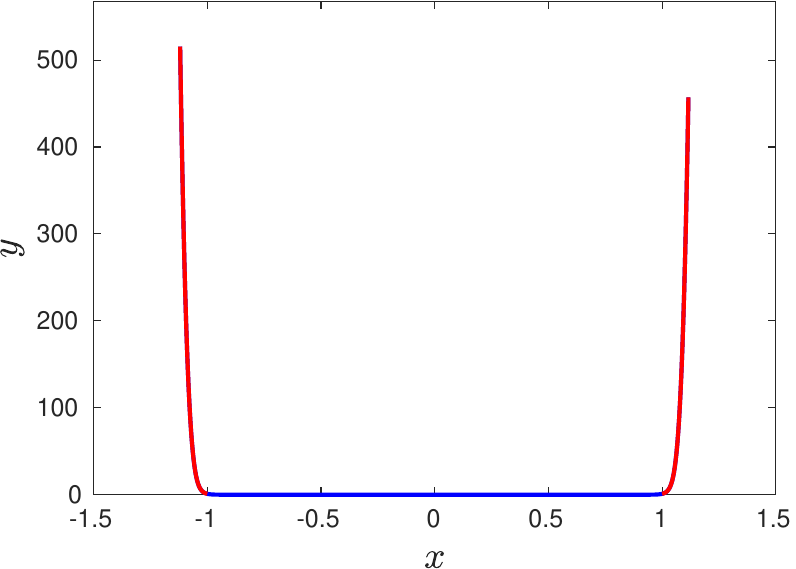}
\hspace{2mm}
\includegraphics[height=33mm]{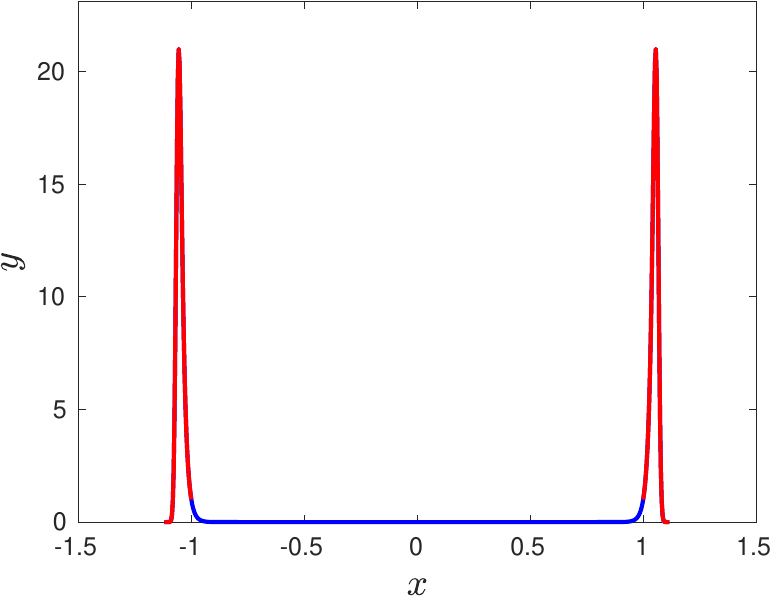}
\hspace{2mm}
\includegraphics[height=33mm]{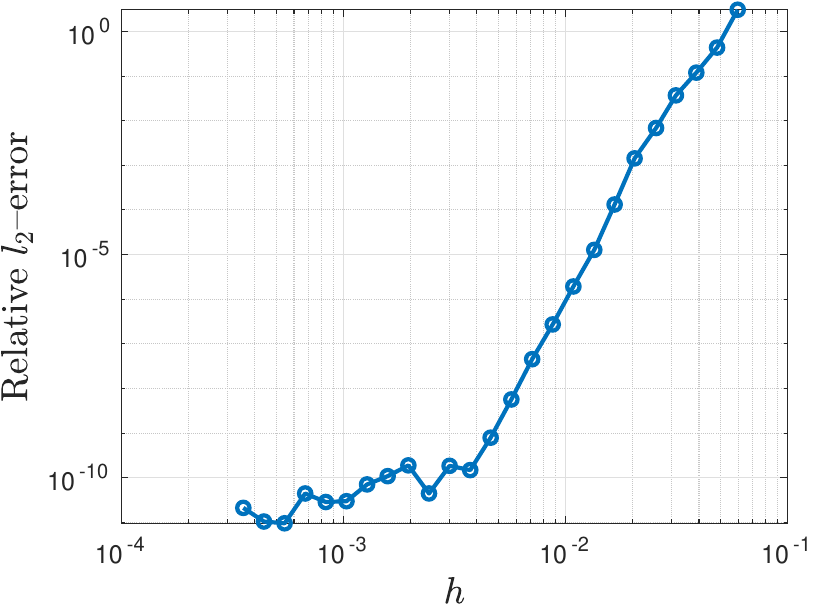}
\caption{\sf Same as \cref{fig1}, but for $f_2(x)=x^2e^{-30(1-x^2)}$.
  $\|E_0[f]\|_\infty/\|f\|_\infty\approx 647$, and
  $\|E_n[f]\|_\infty/\|f\|_\infty\approx 26$.
}
\label{fig2}
\end{figure}

\begin{figure}[!ht]
\centering 
\includegraphics[height=33mm]{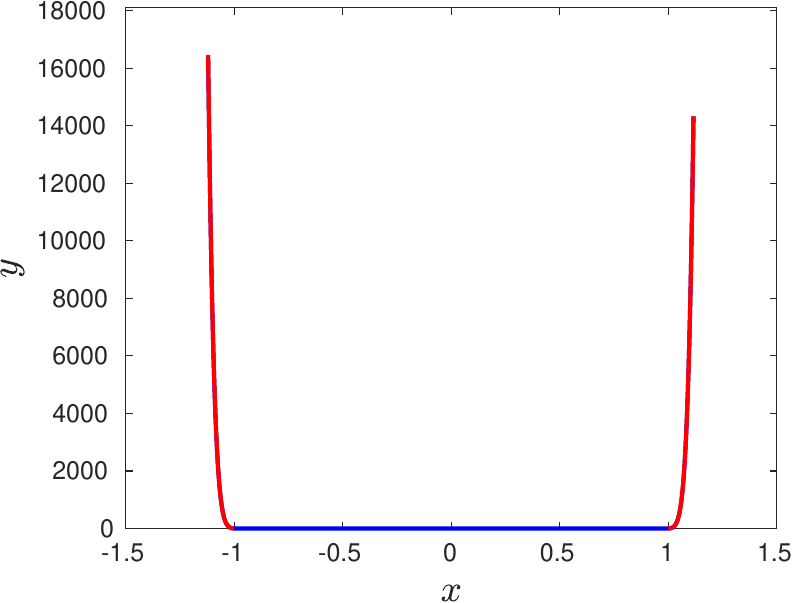}
\hspace{2mm}
\includegraphics[height=33mm]{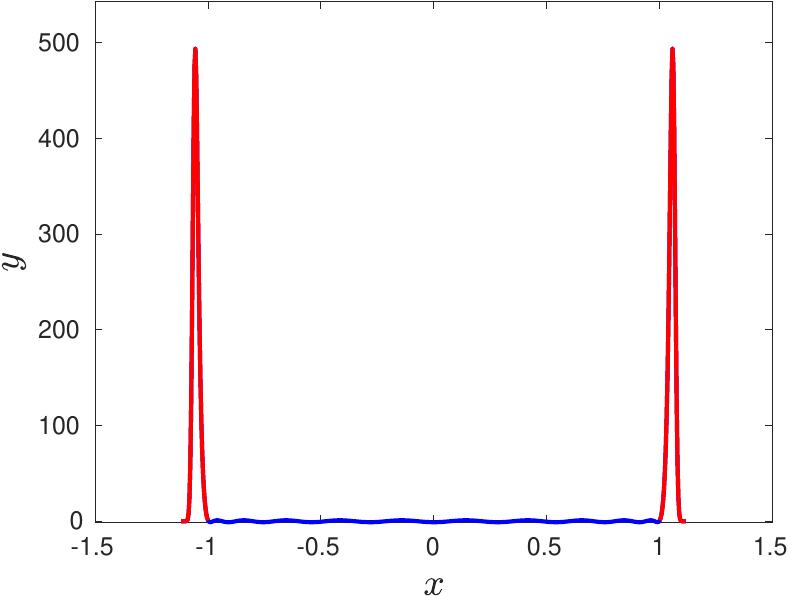}
\hspace{2mm}
\includegraphics[height=33mm]{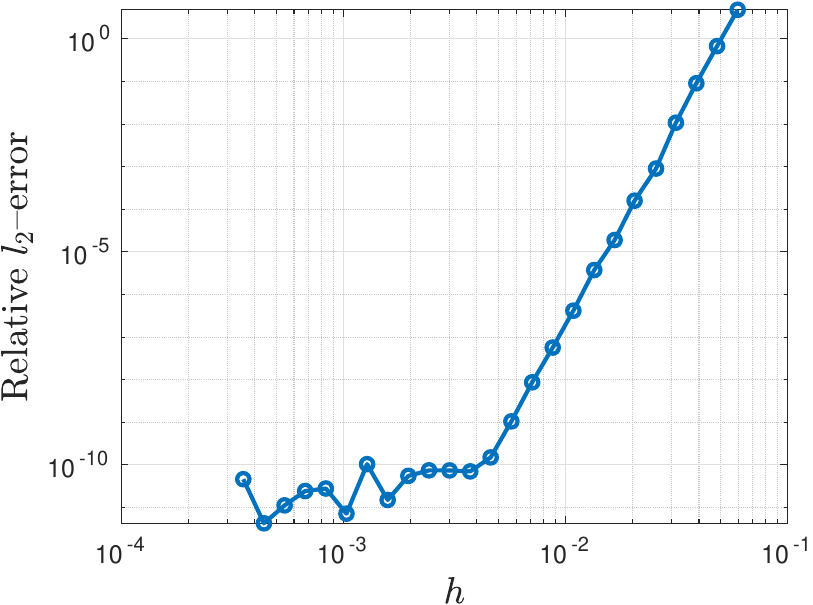}
\caption{\sf Same as \cref{fig1}, but for $f_3(x)=T_{22}(x)$.
  $\|E_0[f]\|_\infty/\|f\|_\infty\approx 16465$, and
  $\|E_n[f]\|_\infty/\|f\|_\infty\approx 494$.
}
\label{fig3}
\end{figure}

We observe that the raw condition number for $f_2$ is greater than that of $f_1$ due to
the existence of a boundary layer in $f_2$, and the raw condition number for $f_3$ is
greater than that of $f_2$ due to the equi-oscillation property of Chebyshev polynomials
on the original interval. For all three functions, the window function is very effective in
controlling the magnitude of the extended part and thus reducing the condition number;
the convergence order is roughly equal to the theoretical value $8$ before the curves
flatten out around $10^{-10}$. 

We would also like to remark that the error curves saturate at $10^{-12}$ or below if we set
$a=2$, clearly showing that the ``intrinsic'' condition number of the function extension
formula does have an effect on the quality of the function extension if it is used
as a standalone tool. However, when function extension is used as an intermediate step
of an FFT-based elliptic BVP solvers, the effect of the higher condition numbers for $a\le 1$
seems to be mitigated by a certain backward stability, which arises as the effect of function
extension is ``subtracted out"  in Step V. As is shown in the next subsection, one may achieve $13$--$14$ digits
of accuracy for the solution to the elliptic BVP for $a=1,$ even though the condition number of $E_8$ approaches $10^6.$

\subsection{Two dimensions -- FFT-based BVP solvers}\label{2dexamples}
In this section, we  use the inhomogeneous Dirichlet problem for the Poisson equation as an example
to illustrate the accuracy and efficiency of the FFT-based BVP solvers when our function
extension scheme is combined with the FFT to generate a particular solution $u^p$. As in the
1D experiments, we set $a=1$
in the formulae~\cref{optnodes}--\cref{optweights} for computing extension nodes and
weights, and the parameters in the window
function~\cref{windowfunc} to $c=40.590000152587891$ (i.e., 16 digits of accuracy),
$r_0=10^{-6}$, $r_1=32h$. For the first three examples, the function extension order is set
to $8$. The error is measured at the points in the uniform grid that are inside the given domain $\Omega$.

In the first example, the boundary $\Gamma_1$
is the unit circle centered at the origin. The inhomogeneous
term in the Poisson equation is
\be
f_1(x)=-100\sin(16\pi x_1)\sin(16\pi x_2), \quad x\in \Omega_1.
\ee
The boundary data is
\be
g_1(x)=50\sin(16\pi x_1)\sin(16\pi x_2)/(16\pi)^2, \quad x\in \Gamma_1,
\ee
and the exact solution $u_1(x)=g_1(x)$ for $x\in \Omega_1$.
The boundary $\Gamma_1$, the right-hand side $f_1$, and its 8th order normal extension $f_1^e$
with $h=10^{-2.5}$ are shown in \cref{example1a}.
\begin{figure}[!ht]
\centering 
\includegraphics[height=50mm]{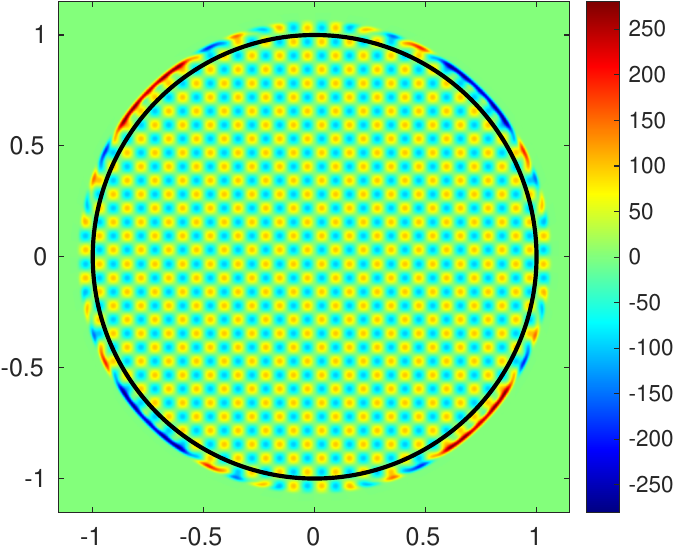}
\caption{\sf The boundary $\Gamma_1$, the right-hand side $f_1$, and its 8th order normal
  extension $f_1^e$ with $h=10^{-2.5}$ for Example 1.
}
\label{example1a}
\end{figure}

\begin{figure}[!ht]
\centering 
\subfloat{{
%
%
\definecolor{mycolor1}{rgb}{0.00000,0.44700,0.74100}%
\definecolor{mycolor2}{rgb}{0.85000,0.32500,0.09800}%
\definecolor{mycolor3}{rgb}{0.92900,0.69400,0.12500}%
\definecolor{mycolor4}{rgb}{0.49400,0.18400,0.55600}%
\definecolor{mycolor5}{rgb}{0.46600,0.67400,0.18800}%
\definecolor{mycolor6}{rgb}{0.00000,0.74902,0.74902}%
\begin{tikzpicture}[scale=0.7]

\begin{loglogaxis}[%
xticklabel style={/pgf/number format/fixed},
xmin=1.4e-3,
xmax=1e-1,
xlabel style={font=\color{white!15!black}},
xlabel={$h$},
xticklabel style = {font=\small},
xminorticks=false,
ymin=1e-15,
ymax=6e4,
yminorticks=true,
ylabel style={font=\small\color{white!15!black}},
ylabel={Relative error},
yticklabel style = {font=\small},
axis background/.style={fill=white},
xmajorgrids,
xminorgrids,
ymajorgrids,
yminorgrids,
legend style={font=\normalsize,at={(1.0,0.195)}, anchor=north east, legend cell align=left, align=left, draw=white!15!black},
clip mode=individual,
]

\addplot [color=mycolor1, line width=1.0pt, mark=*, mark options={solid, mycolor1}, mark size=2.0pt]
  table[row sep=crcr]{%
6.309573e-02 3557.87382064198436637525\\ 
5.411695e-02 1317.32768704955469729612\\ 
4.641589e-02 62.96451397681433803655\\ 
3.981072e-02 4.10179958453420923092\\ 
3.414549e-02 0.15204838183373506344\\ 
2.928645e-02 0.00455308202559867332\\ 
2.511886e-02 0.00022247997606652105\\ 
2.154435e-02 0.00002016773117526561\\ 
1.847850e-02 0.00000284970564475333\\ 
1.584893e-02 0.00000036437604888518\\ 
1.359356e-02 0.00000005738761305440\\ 
1.165914e-02 0.00000001048293105386\\ 
1.000000e-02 0.00000000159864674372\\ 
8.576959e-03 0.00000000044184740209\\ 
7.356423e-03 0.00000000022592347764\\ 
6.309573e-03 0.00000000003276159514\\ 
5.411695e-03 0.00000000003194260371\\ 
4.641589e-03 0.00000000000604736763\\ 
3.981072e-03 0.00000000000012670595\\ 
3.414549e-03 0.00000000000010891508\\ 
2.928645e-03 0.00000000000007488413\\ 
2.511886e-03 0.00000000000006402888\\ 
2.154435e-03 0.00000000000005237124\\ 
1.847850e-03 0.00000000000004028119\\ 
1.584893e-03 0.00000000000003427495\\ 
};
\addlegendentry{$\ell^{2}$}

\addplot [color=black, dashed, line width=1.0pt, mark options={solid, black}]
  table[row sep=crcr]{%
6.309573e-02 475.21246771013284160290\\ 
5.411695e-02 60.24368796187992813884\\ 
4.641589e-02 7.63721952990111851989\\ 
3.981072e-02 0.96818644610221760161\\ 
3.414549e-02 0.12273904013705612404\\ 
2.928645e-02 0.01555988728660148354\\ 
2.511886e-02 0.00197255976665119005\\ 
2.154435e-02 0.00025006556675777206\\ 
1.847850e-02 0.00003170133992139878\\ 
1.584893e-02 0.00000401884580049185\\ 
1.359356e-02 0.00000050947756808313\\ 
1.165914e-02 0.00000006458754708830\\ 
1.000000e-02 0.00000000818789972359\\ 
8.576959e-03 0.00000000103799733704\\ 
7.356423e-03 0.00000000013158911419\\ 
6.309573e-03 0.00000000001668182986\\ 
5.411695e-03 0.00000000000211479080\\ 
4.641589e-03 0.00000000000026809649\\ 
3.981072e-03 0.00000000000003398716\\ 
};
\addlegendentry{$\mathcal{O}(h^{13.5})$ reference}

\end{loglogaxis}
\end{tikzpicture}%

}}%
\subfloat{{
%
%
\definecolor{mycolor1}{rgb}{0.00000,0.44700,0.74100}%
\definecolor{mycolor2}{rgb}{0.85000,0.32500,0.09800}%
\definecolor{mycolor3}{rgb}{0.92900,0.69400,0.12500}%
\definecolor{mycolor4}{rgb}{0.49400,0.18400,0.55600}%
\definecolor{mycolor5}{rgb}{0.46600,0.67400,0.18800}%
\definecolor{mycolor6}{rgb}{0.00000,0.74902,0.74902}%
\begin{tikzpicture}[scale=0.7]

\begin{loglogaxis}[%
xticklabel style={/pgf/number format/fixed},
xmin=110,
xmax=1500,
xlabel style={font=\color{white!15!black}},
xlabel={$N$},
xticklabel style = {font=\small},
xminorticks=false,
ymin=1e-3,
ymax=1e2,
yminorticks=true,
ylabel style={font=\small\color{white!15!black}},
ylabel={Time (sec)},
yticklabel style = {font=\small},
axis background/.style={fill=white},
xmajorgrids,
xminorgrids,
ymajorgrids,
yminorgrids,
legend style={font=\normalsize,at={(1.001,0.33)}, anchor=north east, legend cell align=left, align=left, draw=white!15!black},
clip mode=individual,
]

\addplot [color=mycolor1, line width=1.0pt, mark=*, mark options={solid, mycolor1}, mark size=2.0pt]
  table[row sep=crcr]{%
128 5.70999999999997953637\\ 
134 4.81000000000000227374\\ 
140 4.66000000000002501110\\ 
146 4.50000000000000000000\\ 
156 5.11000000000001364242\\ 
164 4.64999999999997726263\\ 
176 4.24000000000000909495\\ 
190 4.25999999999999090505\\ 
204 4.91000000000002501110\\ 
222 4.98000000000001818989\\ 
244 5.09999999999996589395\\ 
268 5.31000000000000227374\\ 
296 5.59000000000003183231\\ 
330 5.93000000000000682121\\ 
368 6.39999999999997726263\\ 
414 6.83999999999997498890\\ 
466 7.71000000000003637979\\ 
528 9.26999999999998181011\\ 
598 10.07999999999998408384\\ 
682 11.42000000000001591616\\ 
780 13.49000000000000909495\\ 
892 16.07999999999998408384\\ 
1024 18.93999999999999772626\\ 
1178 22.33000000000004092726\\ 
1358 26.44999999999993178790\\ 
};
\addlegendentry{$T_{\rm total}$}

\addplot [color=mycolor2, line width=1.0pt, mark=diamond, mark options={solid, mycolor2}, mark size=2.0pt]
  table[row sep=crcr]{%
128 0.18999999999994088284\\ 
134 0.12000000000000454747\\ 
140 0.08000000000004092726\\ 
146 0.06999999999999317879\\ 
156 0.06999999999999317879\\ 
164 0.08000000000004092726\\ 
176 0.04999999999995452526\\ 
190 0.03999999999996362021\\ 
204 0.14999999999997726263\\ 
222 0.06999999999999317879\\ 
244 0.09000000000003183231\\ 
268 0.10000000000002273737\\ 
296 0.10999999999995679900\\ 
330 0.08000000000004092726\\ 
368 0.12999999999999545253\\ 
414 0.12999999999999545253\\ 
466 0.19999999999998863132\\ 
528 0.20999999999997953637\\ 
598 0.21999999999997044142\\ 
682 0.33999999999997498890\\ 
780 0.31000000000000227374\\ 
892 0.45999999999997953637\\ 
1024 0.74000000000000909495\\ 
1178 0.87000000000000454747\\ 
1358 1.81999999999993633537\\ 
};
\addlegendentry{$T_{\rm FFT}$}

\addplot [color=mycolor3, line width=1.0pt, mark=square, mark options={solid, mycolor3}, mark size=2.0pt]
  table[row sep=crcr]{%
128 0.01272799999999999973\\ 
134 0.01037399999999999954\\ 
140 0.01089299999999999990\\ 
146 0.01137499999999999969\\ 
156 0.01280300000000000014\\ 
164 0.01375600000000000087\\ 
176 0.01679500000000000090\\ 
190 0.01660899999999999876\\ 
204 0.01915299999999999989\\ 
222 0.02078800000000000092\\ 
244 0.02571100000000000121\\ 
268 0.02627799999999999928\\ 
296 0.03022000000000000033\\ 
330 0.03385900000000000021\\ 
368 0.03869199999999999723\\ 
414 0.04531999999999999917\\ 
466 0.05301100000000000256\\ 
528 0.06380900000000000460\\ 
598 0.07161900000000000210\\ 
682 0.08259400000000000075\\ 
780 0.08949100000000000110\\ 
892 0.11232300000000000617\\ 
1024 0.13175999999999998824\\ 
1178 0.15344199999999999506\\ 
1358 0.18036499999999999755\\ 
};
\addlegendentry{$T_{\rm extension}$}

\addplot [color=mycolor4, line width=1.0pt, mark=star, mark options={solid, mycolor4}, mark size=2.0pt]
  table[row sep=crcr]{%
128 1.67606400000000466122\\ 
134 1.50845100000001819396\\ 
140 1.50953999999996124615\\ 
146 1.39894600000000446371\\ 
156 1.54415199999999086522\\ 
164 1.37856800000002266415\\ 
176 1.16933900000004320496\\ 
190 1.15231500000000464112\\ 
204 1.36570199999999308815\\ 
222 1.39386199999997506538\\ 
244 1.29519900000000909301\\ 
268 1.27279600000003645377\\ 
296 1.38071400000000221731\\ 
330 1.43231300000002281791\\ 
368 1.51765000000001593072\\ 
414 1.62291899999999089665\\ 
466 1.89331100000004770578\\ 
528 2.12417000000001587878\\ 
598 2.17407400000001604923\\ 
682 2.43709900000000434517\\ 
780 2.48766700000000229309\\ 
892 2.96645600000002485075\\ 
1024 3.46719799999999800377\\ 
1178 3.78192599999999323757\\ 
1358 3.99974899999994537581\\ 
};
\addlegendentry{$T_{\rm sort}$}

\end{loglogaxis}
\end{tikzpicture}%

}}%
\caption{\sf Left: convergence order study for Example 1. Data points are shown in small
  circles, while the dashed line is the least squares fitting of the first $19$ data points,
  leading to an estimated convergence order $13.5$. The x-axis is
  the step size of the equispaced grid. Right: timing results for Example 1, where
  the x-axis is the number of points along each dimension and the y-axis shows
  the computational time of various steps in seconds. The total number
  of equispaced grid points in the volume is $N^2$.
}
\label{example1b}
\end{figure}
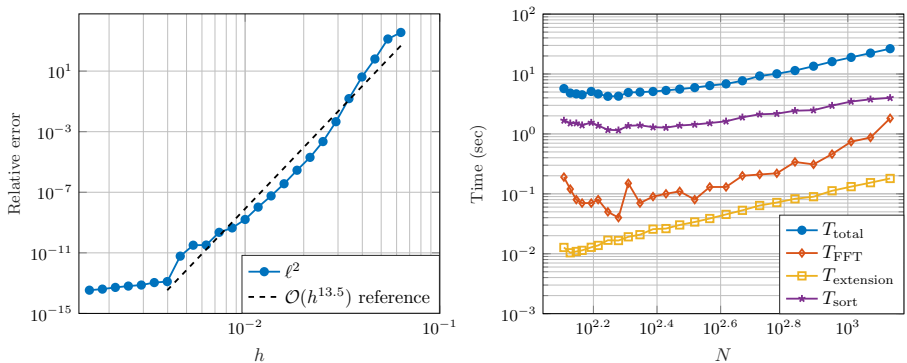

\begin{figure}[!ht]
\centering 
\includegraphics[height=50mm]{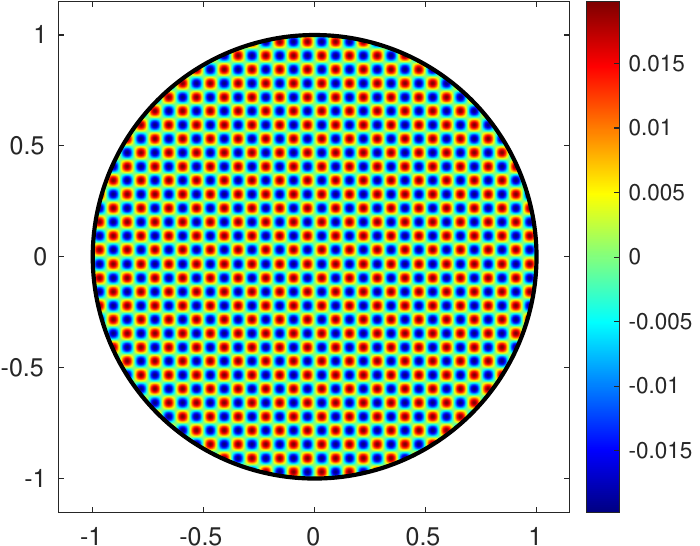}
\hspace{4mm}
\includegraphics[height=50mm]{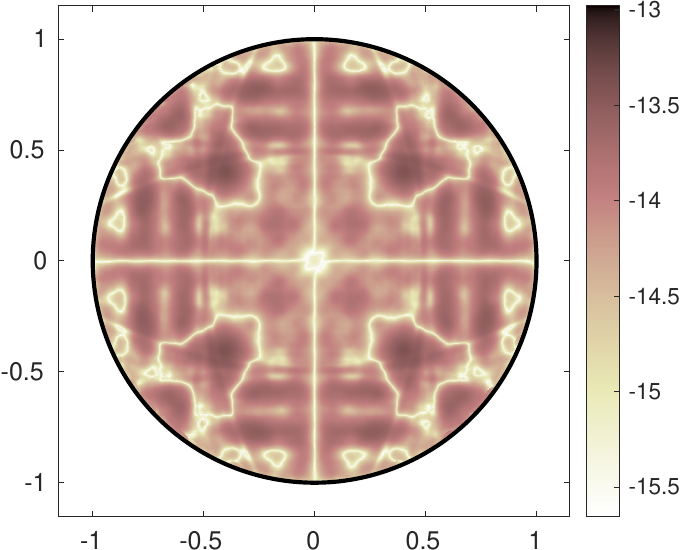}
\caption{\sf Left: solution to the Dirichlet Poisson problem in Example 1.
  Right: $\log_{10}$ of absolute error in
  the solution in $\Omega$. Here, the total number
  of uniform grid points is $728^2$. The relative $l^2$ error is about $9\times 10^{-13}$
  on the grid.
}
\label{example1c}
\end{figure}

We set the spacing $h$ of the uniform grid to $25$ logarithmically
equally spaced points between $10^{-1.2}$ and $10^{-2.8}$. The number of points $N$
along each dimension varies from $128$ to $1358$, correspondingly. The convergence order and timing
results (in seconds) are shown in \cref{example1b}. For the figure on the right side of \cref{example1b},
$T_{\rm total}$ is the total computational time; $T_{\rm FFT}$ is the computational time on FFT;
$T_{\rm extension}$ is the computational time on function extension; and $T_{\rm sort}$ is the time labeling
points as inside and outside (for points outside and within a distance $r_{1}$ of the boundary,
we also need to find the closest point on the boundary). We observe that the convergence order is
about $13.5$, much higher than the theoretical value $10$, which is probably due to the
simple geometry of the boundary curve. As to the timing results, the time on function extension
increases slower than the time on FFT as $N$ increases. It is clear that
other parts of the scheme such as the FFT, sorting points and the FMM dominate the total time for
large $N$, and the cost on function extension is negligible. The solution to the Dirichlet Poisson
problem and its numerical error are shown in \cref{example1c}.

In the second example, the boundary $\Gamma_2$ consists of two star fish.
The inhomogeneous
term in the Poisson equation is
\be
f_2(x)=-200\sin(10(x_1+x_2))+\frac{2}{9}+1000(1000x_1^2-1)e^{-500x_1^2}, \quad x\in \Omega_2.
\ee
The boundary data is
\be
g_2(x)=\sin(10(x_1+x_2))+\frac{x_1^2}{9}-x_2+8+e^{-500x_1^2}, \quad x\in \Gamma_2,
\ee
and the exact solution $u_2(x)=g_2(x)$ for $x\in \Omega_2$.
The boundary $\Gamma_2$, the right-hand side $f_2$, and its 8th order normal extension $f_2^e$
with $h=2\times 10^{-3}$ are shown in \cref{example2a}.

\begin{figure}[!ht]
\centering 
\includegraphics[height=50mm]{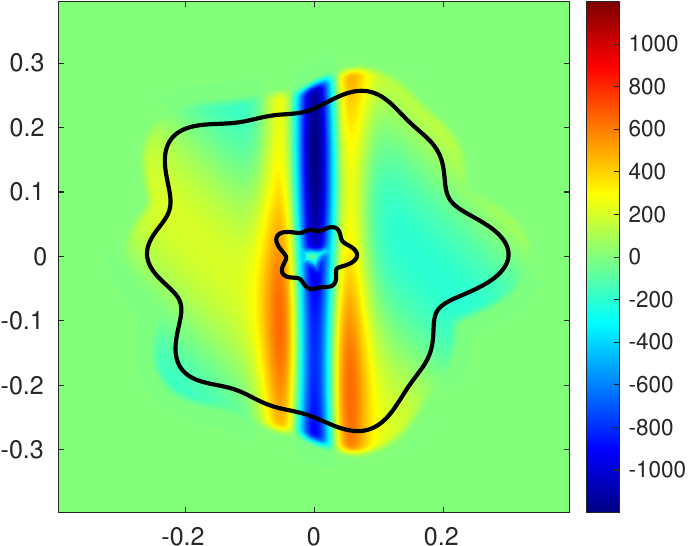}
\caption{\sf The boundary $\Gamma_2$, the right-hand side $f_2$, and its 8th order normal
  extension $f_2^e$ with $h=2\times 10^{-3}$ for Example 2.
}
\label{example2a}
\end{figure}

\begin{figure}[!ht]
\centering 
\subfloat{{
%
%
\definecolor{mycolor1}{rgb}{0.00000,0.44700,0.74100}%
\definecolor{mycolor2}{rgb}{0.85000,0.32500,0.09800}%
\definecolor{mycolor3}{rgb}{0.92900,0.69400,0.12500}%
\definecolor{mycolor4}{rgb}{0.49400,0.18400,0.55600}%
\definecolor{mycolor5}{rgb}{0.46600,0.67400,0.18800}%
\definecolor{mycolor6}{rgb}{0.00000,0.74902,0.74902}%
\begin{tikzpicture}[scale=0.7]

\begin{loglogaxis}[%
xticklabel style={/pgf/number format/fixed},
xmin=1e-4,
xmax=1e-1,
xlabel style={font=\color{white!15!black}},
xlabel={$h$},
xticklabel style = {font=\small},
xminorticks=false,
ymin=1e-15,
ymax=1e1,
yminorticks=true,
ylabel style={font=\small\color{white!15!black}},
ylabel={Relative error},
yticklabel style = {font=\small},
axis background/.style={fill=white},
xmajorgrids,
xminorgrids,
ymajorgrids,
yminorgrids,
legend style={font=\normalsize,at={(1.0,0.195)}, anchor=north east, legend cell align=left, align=left, draw=white!15!black},
clip mode=individual,
]

\addplot [color=mycolor1, line width=1.0pt, mark=*, mark options={solid, mycolor1}, mark size=2.0pt]
  table[row sep=crcr]{%
3.981072e-02 0.50319639734929000330\\ 
3.254618e-02 0.19654481874398779806\\ 
2.660725e-02 0.03547796390922491883\\ 
2.175204e-02 0.02885625504051265999\\ 
1.778279e-02 0.00764003960411458859\\ 
1.453784e-02 0.00297437641091006849\\ 
1.188502e-02 0.00031447018357088789\\ 
9.716280e-03 0.00004905996301471407\\ 
7.943282e-03 0.00000474347915162700\\ 
6.493816e-03 0.00000063216435676157\\ 
5.308844e-03 0.00000006053495003017\\ 
4.340103e-03 0.00000000766979757826\\ 
3.548134e-03 0.00000000441867998377\\ 
2.900681e-03 0.00000000267352349675\\ 
2.371374e-03 0.00000000046076296439\\ 
1.938653e-03 0.00000000015265399373\\ 
1.584893e-03 0.00000000008130507445\\ 
1.295687e-03 0.00000000002996340517\\ 
1.059254e-03 0.00000000002130957618\\ 
8.659643e-04 0.00000000000696509075\\ 
7.079458e-04 0.00000000000443969314\\ 
5.787620e-04 0.00000000000085932771\\ 
4.731513e-04 0.00000000000010310265\\ 
3.868121e-04 0.00000000000000658649\\ 
3.162278e-04 0.00000000000000331793\\ 
};
\addlegendentry{$\ell^{2}$}

\addplot [color=black, dashed, line width=1.0pt, mark options={solid, black}]
  table[row sep=crcr]{%
3.981072e-02 0.44225807012676809915\\ 
3.254618e-02 0.11277757347744089833\\ 
2.660725e-02 0.02875873147055021950\\ 
2.175204e-02 0.00733359133640747868\\ 
1.778279e-02 0.00187009506815363291\\ 
1.453784e-02 0.00047688170822534010\\ 
1.188502e-02 0.00012160673941803928\\ 
9.716280e-03 0.00003101020403344744\\ 
7.943282e-03 0.00000790772582833840\\ 
6.493816e-03 0.00000201650165567190\\ 
5.308844e-03 0.00000051421597253100\\ 
4.340103e-03 0.00000013112712586288\\ 
3.548134e-03 0.00000003343794058444\\ 
2.900681e-03 0.00000000852680834092\\ 
2.371374e-03 0.00000000217437016790\\ 
1.938653e-03 0.00000000055447307340\\ 
1.584893e-03 0.00000000014139284730\\ 
1.295687e-03 0.00000000003605574053\\ 
1.059254e-03 0.00000000000919435778\\ 
8.659643e-04 0.00000000000234459794\\ 
7.079458e-04 0.00000000000059788183\\ 
5.787620e-04 0.00000000000015246225\\ 
4.731513e-04 0.00000000000003887848\\ 
3.868121e-04 0.00000000000000991417\\ 
3.162278e-04 0.00000000000000252815\\ 
};
\addlegendentry{$\mathcal{O}(h^{6.8})$ reference}

\end{loglogaxis}
\end{tikzpicture}%

}}%
\hspace{4mm}
\subfloat{{
%
%
\definecolor{mycolor1}{rgb}{0.00000,0.44700,0.74100}%
\definecolor{mycolor2}{rgb}{0.85000,0.32500,0.09800}%
\definecolor{mycolor3}{rgb}{0.92900,0.69400,0.12500}%
\definecolor{mycolor4}{rgb}{0.49400,0.18400,0.55600}%
\definecolor{mycolor5}{rgb}{0.46600,0.67400,0.18800}%
\definecolor{mycolor6}{rgb}{0.00000,0.74902,0.74902}%
\begin{tikzpicture}[scale=0.7]

\begin{loglogaxis}[%
xticklabel style={/pgf/number format/fixed},
xmin=110,
xmax=1500,
xlabel style={font=\color{white!15!black}},
xlabel={$N$},
xticklabel style = {font=\small},
xminorticks=false,
ymin=1e-3,
ymax=1e2,
yminorticks=true,
ylabel style={font=\small\color{white!15!black}},
ylabel={Time (sec)},
yticklabel style = {font=\small},
axis background/.style={fill=white},
xmajorgrids,
xminorgrids,
ymajorgrids,
yminorgrids,
legend style={font=\normalsize,at={(1.001,0.33)}, anchor=north east, legend cell align=left, align=left, draw=white!15!black},
clip mode=individual,
]

\addplot [color=mycolor1, line width=1.0pt, mark=*, mark options={solid, mycolor1}, mark size=2.0pt]
  table[row sep=crcr]{%
112 6.99000000000000909495\\ 
114 6.36000000000001364242\\ 
120 5.99000000000000909495\\ 
124 7.25000000000000000000\\ 
130 6.90000000000009094947\\ 
138 6.67999999999994997779\\ 
148 6.75999999999999090505\\ 
158 6.52999999999997271516\\ 
172 7.14000000000010004442\\ 
188 6.76999999999998181011\\ 
210 8.06999999999993633537\\ 
234 7.75999999999999090505\\ 
266 8.13999999999998635758\\ 
304 8.35999999999989995558\\ 
350 9.15000000000009094947\\ 
406 9.48000000000001818989\\ 
476 10.41999999999995907274\\ 
560 11.49000000000000909495\\ 
664 13.23000000000001818989\\ 
790 15.75999999999999090505\\ 
944 19.21000000000003637979\\ 
1134 22.25000000000000000000\\ 
1366 27.00999999999999090505\\ 
1648 33.50999999999999090505\\ 
1994 41.68999999999994088284\\ 
};
\addlegendentry{$T_{\rm total}$}

\addplot [color=mycolor2, line width=1.0pt, mark=diamond, mark options={solid, mycolor2}, mark size=2.0pt]
  table[row sep=crcr]{%
112 0.12000000000000454747\\ 
114 0.04999999999995452526\\ 
120 0.05999999999994543032\\ 
124 0.08000000000004092726\\ 
130 0.08000000000004092726\\ 
138 0.08000000000004092726\\ 
148 0.07000000000005002221\\ 
158 0.12999999999999545253\\ 
172 0.06999999999993633537\\ 
188 0.05000000000006821210\\ 
210 0.06000000000005911716\\ 
234 0.07999999999992724042\\ 
266 0.09000000000003183231\\ 
304 0.14999999999997726263\\ 
350 0.15999999999996816769\\ 
406 0.12999999999999545253\\ 
476 0.12000000000000454747\\ 
560 0.15999999999996816769\\ 
664 0.24000000000000909495\\ 
790 0.35000000000002273737\\ 
944 0.67000000000007275958\\ 
1134 0.92000000000007275958\\ 
1366 1.20000000000004547474\\ 
1648 1.99000000000000909495\\ 
1994 2.79999999999995452526\\ 
};
\addlegendentry{$T_{\rm FFT}$}

\addplot [color=mycolor3, line width=1.0pt, mark=square, mark options={solid, mycolor3}, mark size=2.0pt]
  table[row sep=crcr]{%
112 0.01674900000000000000\\ 
114 0.01248299999999999930\\ 
120 0.01362799999999999949\\ 
124 0.01533900000000000020\\ 
130 0.01091899999999999989\\ 
138 0.01734000000000000125\\ 
148 0.01945300000000000154\\ 
158 0.03091199999999999851\\ 
172 0.02798300000000000093\\ 
188 0.01804800000000000154\\ 
210 0.02512999999999999956\\ 
234 0.03615199999999999664\\ 
266 0.04384500000000000203\\ 
304 0.04809100000000000180\\ 
350 0.04351399999999999713\\ 
406 0.05035400000000000292\\ 
476 0.08052600000000000036\\ 
560 0.07561900000000000566\\ 
664 0.09207799999999999319\\ 
790 0.10923900000000000277\\ 
944 0.12670999999999998931\\ 
1134 0.15037300000000000666\\ 
1366 0.18577899999999999969\\ 
1648 0.22434999999999999387\\ 
1994 0.27268199999999997996\\ 
};
\addlegendentry{$T_{\rm extension}$}

\addplot [color=mycolor4, line width=1.0pt, mark=star, mark options={solid, mycolor4}, mark size=2.0pt]
  table[row sep=crcr]{%
112 1.98073099999994095555\\ 
114 1.61747200000003177323\\ 
120 1.63198700000001362120\\ 
124 1.96234399999994990615\\ 
130 1.87214999999992715018\\ 
138 1.94741999999996817650\\ 
148 1.93559799999998638498\\ 
158 1.56162099999999082200\\ 
172 2.17264300000000476487\\ 
188 1.91924699999999548439\\ 
210 2.29176700000001831725\\ 
234 2.04417599999991361770\\ 
266 1.93259000000004088626\\ 
304 2.14234300000008648368\\ 
350 2.15263400000006832613\\ 
406 2.38629299999999089366\\ 
476 2.54521600000002257147\\ 
560 2.56828000000004097458\\ 
664 3.01202800000003634295\\ 
790 3.49309100000000904984\\ 
944 4.07666999999998669324\\ 
1134 3.97579700000004088167\\ 
1366 4.31195400000002315721\\ 
1648 4.82978399999990948288\\ 
1994 5.04299600000000403099\\ 
};
\addlegendentry{$T_{\rm sort}$}

\end{loglogaxis}
\end{tikzpicture}%

}}%
\caption{\sf Left: convergence order study for Example 2. Data points are shown in small
  circles, while the dashed line is the least squares fitting of all $25$ data points,
  leading to an estimated convergence order $6.8$. Right: timing results for Example 2.
}
\label{example2b}
\end{figure}

\begin{figure}[!ht]
\centering 
\includegraphics[height=50mm]{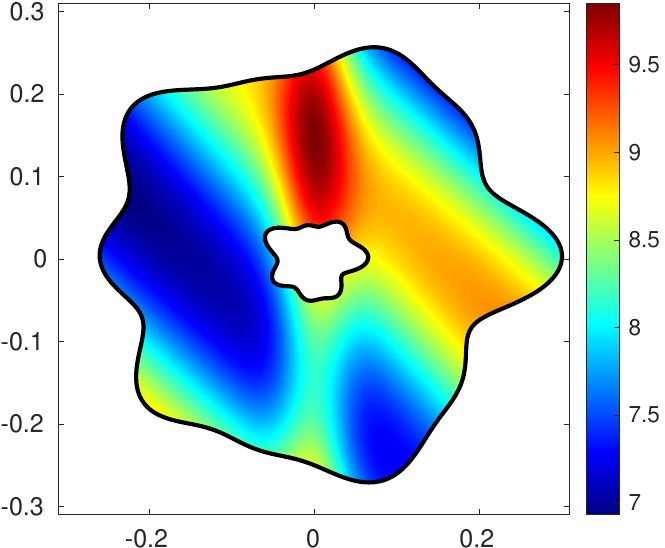}
\hspace{4mm}
\includegraphics[height=50mm]{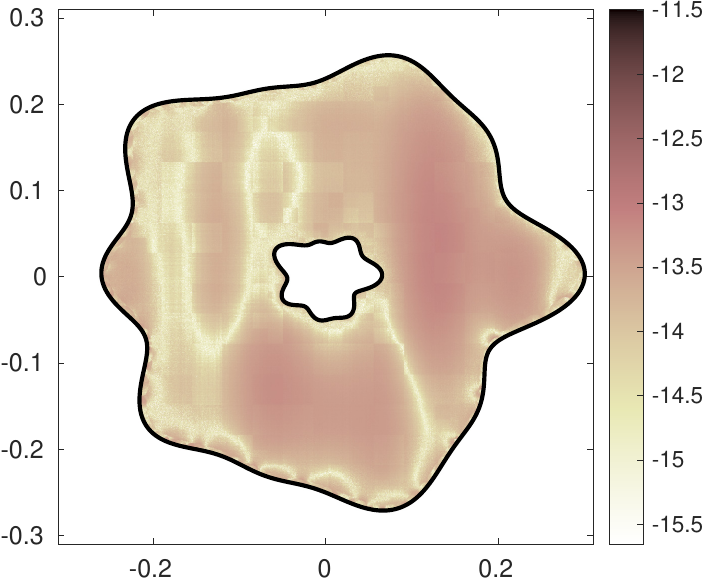}
\caption{\sf Left: solution to the Dirichlet Poisson problem in Example 2.
  Right: $\log_{10}$ of absolute error in the solution in $\Omega$. Here, the total number
  of uniform grid points is $3098^2$. The relative $l^2$ error is about $2\times 10^{-15}$
  on the grid.
}
\label{example2c}
\end{figure}

We set the spacing $h$ of the uniform grid to $25$ logarithmically
equally spaced points between $10^{-1.4}$ and $10^{-3.5}$. The number of points $N$
along each dimension varies from $112$ to $1994$, correspondingly. The convergence
order and timing results are shown in \cref{example2b}. We observe that the convergence
order is about $6.8$, and that the timing results exhibit similar pattern as Example 1.
The solution to the Dirichlet Poisson problem and its numerical error 
are shown in \cref{example2c}.

In the third example, the boundary $\Gamma_3$ consists of two star fish.
The inhomogeneous term in the Poisson equation is
\be
f_3(x)=-\sum_{k=0}^5e^{-\sqrt{2^k}}2^{2k}(\cos(2^kx_1)+\cos(2^kx_2)), \quad x\in \Omega_3.
\ee
The boundary data is
\be
g_3(x)=\sum_{k=0}^5e^{-\sqrt{2^k}}(\cos(2^kx_1)+\cos(2^kx_2)), \quad x\in \Omega_3,
\ee
and the exact solution $u_3(x)=g_3(x)$ for $x\in \Omega_3$.
The boundary $\Gamma_3$, the right-hand side $f_3$, and its 8th order normal extension $f_3^e$
with $h=10^{-2.2}$ are shown in \cref{example3a}.

\begin{figure}[!ht]
\centering 
\includegraphics[height=50mm]{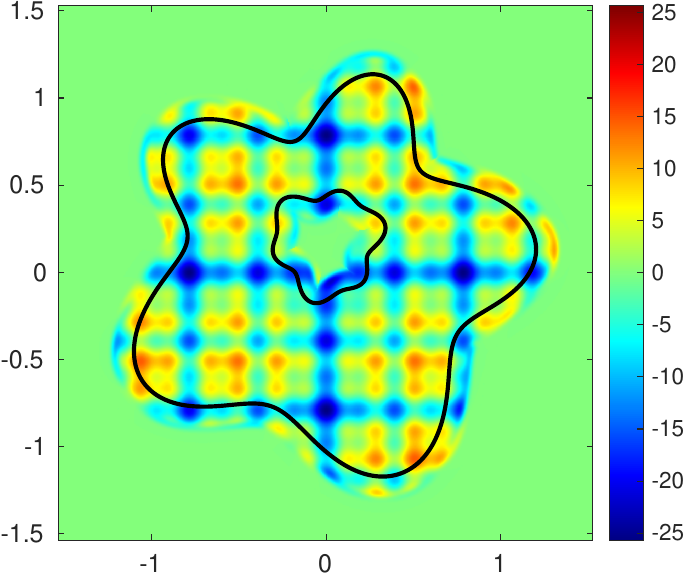}
\caption{\sf The boundary $\Gamma_3$, the right-hand side $f_3$, and its 8th order normal
  extension $f_3^e$ with $h=10^{-2.2}$ for Example 3.
}
\label{example3a}
\end{figure}

\begin{figure}[!ht]
\centering 
\subfloat{{
%
%
\definecolor{mycolor1}{rgb}{0.00000,0.44700,0.74100}%
\definecolor{mycolor2}{rgb}{0.85000,0.32500,0.09800}%
\definecolor{mycolor3}{rgb}{0.92900,0.69400,0.12500}%
\definecolor{mycolor4}{rgb}{0.49400,0.18400,0.55600}%
\definecolor{mycolor5}{rgb}{0.46600,0.67400,0.18800}%
\definecolor{mycolor6}{rgb}{0.00000,0.74902,0.74902}%
\begin{tikzpicture}[scale=0.7]

\begin{loglogaxis}[%
xticklabel style={/pgf/number format/fixed},
xmin=0.7e-3,
xmax=1e-1,
xlabel style={font=\color{white!15!black}},
xlabel={$h$},
xticklabel style = {font=\small},
xminorticks=false,
ymin=1e-16,
ymax=1e1,
yminorticks=true,
ylabel style={font=\small\color{white!15!black}},
ylabel={Relative error},
yticklabel style = {font=\small},
axis background/.style={fill=white},
xmajorgrids,
xminorgrids,
ymajorgrids,
yminorgrids,
legend style={font=\normalsize,at={(1.0,0.195)}, anchor=north east, legend cell align=left, align=left, draw=white!15!black},
clip mode=individual,
]

\addplot [color=mycolor1, line width=1.0pt, mark=*, mark options={solid, mycolor1}, mark size=2.0pt]
  table[row sep=crcr]{%
7.943282e-02 0.46165549041917119943\\ 
6.619624e-02 0.24711699180558899247\\ 
5.516539e-02 0.12055478870129658631\\ 
4.597270e-02 0.07440524584290744703\\ 
3.831187e-02 0.01743188719062399583\\ 
3.192763e-02 0.00587861083314751565\\ 
2.660725e-02 0.00348998941874827063\\ 
2.217345e-02 0.00029458899109021335\\ 
1.847850e-02 0.00009587977598246820\\ 
1.539927e-02 0.00002178651794481234\\ 
1.283315e-02 0.00000707122018481544\\ 
1.069465e-02 0.00000085722918587561\\ 
8.912509e-03 0.00000010565500037875\\ 
7.427340e-03 0.00000001352112935659\\ 
6.189658e-03 0.00000000135868229869\\ 
5.158222e-03 0.00000000022868495775\\ 
4.298662e-03 0.00000000002244769231\\ 
3.582339e-03 0.00000000000038044661\\ 
2.985383e-03 0.00000000000000482812\\ 
2.487902e-03 0.00000000000000317257\\ 
2.073322e-03 0.00000000000001468151\\ 
1.727826e-03 0.00000000000001411095\\ 
1.439903e-03 0.00000000000002667139\\ 
1.199960e-03 0.00000000000002798281\\ 
1.000000e-03 0.00000000000002635625\\ 
};
\addlegendentry{$\ell^{2}$}

\addplot [color=black, dashed, line width=1.0pt, mark options={solid, black}]
  table[row sep=crcr]{%
7.943282e-02 19.27379448588492039107\\ 
6.619624e-02 3.37383445935966719276\\ 
5.516539e-02 0.59058214860072588692\\ 
4.597270e-02 0.10338007938660023766\\ 
3.831187e-02 0.01809645083127158632\\ 
3.192763e-02 0.00316774309549502225\\ 
2.660725e-02 0.00055450631798563106\\ 
2.217345e-02 0.00009706508621966713\\ 
1.847850e-02 0.00001699102545315926\\ 
1.539927e-02 0.00000297424086449128\\ 
1.283315e-02 0.00000052063418681802\\ 
1.069465e-02 0.00000009113584569420\\ 
8.912509e-03 0.00000001595312520901\\ 
7.427340e-03 0.00000000279255875661\\ 
6.189658e-03 0.00000000048883114167\\ 
5.158222e-03 0.00000000008556879403\\ 
4.298662e-03 0.00000000001497862531\\ 
3.582339e-03 0.00000000000262197474\\ 
2.985383e-03 0.00000000000045897079\\ 
2.487902e-03 0.00000000000008034181\\ 
};
\addlegendentry{$\mathcal{O}(h^{9.6})$ reference}

\end{loglogaxis}
\end{tikzpicture}%

}}%
\hspace{4mm}
\subfloat{{
%
%
\definecolor{mycolor1}{rgb}{0.00000,0.44700,0.74100}%
\definecolor{mycolor2}{rgb}{0.85000,0.32500,0.09800}%
\definecolor{mycolor3}{rgb}{0.92900,0.69400,0.12500}%
\definecolor{mycolor4}{rgb}{0.49400,0.18400,0.55600}%
\definecolor{mycolor5}{rgb}{0.46600,0.67400,0.18800}%
\definecolor{mycolor6}{rgb}{0.00000,0.74902,0.74902}%
\begin{tikzpicture}[scale=0.7]

\begin{loglogaxis}[%
xticklabel style={/pgf/number format/fixed},
xmin=110,
xmax=3000,
xlabel style={font=\color{white!15!black}},
xlabel={$N$},
xticklabel style = {font=\small},
xminorticks=false,
ymin=1e-3,
ymax=1e2,
yminorticks=true,
ylabel style={font=\small\color{white!15!black}},
ylabel={Time (sec)},
yticklabel style = {font=\small},
axis background/.style={fill=white},
xmajorgrids,
xminorgrids,
ymajorgrids,
yminorgrids,
legend style={font=\normalsize,at={(1.001,0.33)}, anchor=north east, legend cell align=left, align=left, draw=white!15!black},
clip mode=individual,
]

\addplot [color=mycolor1, line width=1.0pt, mark=*, mark options={solid, mycolor1}, mark size=2.0pt]
  table[row sep=crcr]{%
128 28.24000000000000909495\\ 
134 27.35000000000002273737\\ 
142 27.42999999999994997779\\ 
150 27.66000000000008185452\\ 
160 27.46000000000003637979\\ 
174 27.62999999999988176569\\ 
188 26.96000000000003637979\\ 
208 28.31999999999993633537\\ 
230 28.00000000000000000000\\ 
256 28.49000000000000909495\\ 
288 30.69000000000005456968\\ 
326 29.91000000000008185452\\ 
372 30.15999999999985448085\\ 
428 30.91000000000008185452\\ 
494 31.31999999999993633537\\ 
574 33.04000000000019099389\\ 
668 34.76999999999998181011\\ 
784 38.93999999999982719601\\ 
920 40.99000000000000909495\\ 
1086 47.15000000000009094947\\ 
1282 51.16000000000008185452\\ 
1520 56.69999999999981810106\\ 
1804 66.40000000000009094947\\ 
2146 77.33999999999991814548\\ 
2556 92.06000000000017280399\\ 
};
\addlegendentry{$T_{\rm total}$}

\addplot [color=mycolor2, line width=1.0pt, mark=diamond, mark options={solid, mycolor2}, mark size=2.0pt]
  table[row sep=crcr]{%
128 0.14000000000010004442\\ 
134 0.07999999999992724042\\ 
142 0.10000000000002273737\\ 
150 0.07999999999992724042\\ 
160 0.08999999999991814548\\ 
174 0.08999999999991814548\\ 
188 0.06999999999993633537\\ 
208 0.07999999999992724042\\ 
230 0.09999999999990905053\\ 
256 0.09999999999990905053\\ 
288 0.11999999999989086064\\ 
326 0.13999999999987267074\\ 
372 0.15000000000009094947\\ 
428 0.17999999999983629095\\ 
494 0.18000000000006366463\\ 
574 0.17000000000007275958\\ 
668 0.41000000000008185452\\ 
784 0.29999999999995452526\\ 
920 0.57999999999992724042\\ 
1086 0.98000000000001818989\\ 
1282 1.01999999999998181011\\ 
1520 1.91000000000008185452\\ 
1804 2.59999999999990905053\\ 
2146 3.33000000000015461410\\ 
2556 4.35000000000013642421\\ 
};
\addlegendentry{$T_{\rm FFT}$}

\addplot [color=mycolor3, line width=1.0pt, mark=square, mark options={solid, mycolor3}, mark size=2.0pt]
  table[row sep=crcr]{%
128 0.03975100000000000161\\ 
134 0.03124800000000000147\\ 
142 0.02214999999999999955\\ 
150 0.02344800000000000009\\ 
160 0.02672100000000000170\\ 
174 0.04814500000000000030\\ 
188 0.03340100000000000013\\ 
208 0.04110199999999999965\\ 
230 0.04899099999999999983\\ 
256 0.05619299999999999989\\ 
288 0.06794799999999999451\\ 
326 0.07989100000000000368\\ 
372 0.09033499999999999863\\ 
428 0.10486800000000000288\\ 
494 0.12519500000000000073\\ 
574 0.14286599999999999300\\ 
668 0.17265800000000000591\\ 
784 0.21951499999999998791\\ 
920 0.26373499999999999721\\ 
1086 0.31860699999999997356\\ 
1282 0.38562899999999999956\\ 
1520 0.45442900000000002736\\ 
1804 0.53817899999999996297\\ 
2146 0.69325099999999995060\\ 
2556 0.82241200000000003190\\ 
};
\addlegendentry{$T_{\rm extension}$}

\addplot [color=mycolor4, line width=1.0pt, mark=star, mark options={solid, mycolor4}, mark size=2.0pt]
  table[row sep=crcr]{%
128 2.97836800000009116474\\ 
134 3.17407100000000452766\\ 
142 3.00397499999993167208\\ 
150 3.37576600000004090063\\ 
160 3.13925599999991833400\\ 
174 3.04143700000000904993\\ 
188 2.84140800000019089211\\ 
208 2.99241000000006351911\\ 
230 3.28221200000003632269\\ 
256 3.05045200000003635665\\ 
288 3.11648699999999090338\\ 
326 3.30347800000007296717\\ 
372 3.32612300000007277134\\ 
428 3.42978000000000005087\\ 
494 3.56604799999989996451\\ 
574 3.96975500000008185708\\ 
668 4.10007700000005481655\\ 
784 4.66957600000006323171\\ 
920 4.80627499999993634816\\ 
1086 5.01096999999995418307\\ 
1282 5.32682000000016397934\\ 
1520 5.85804999999986364401\\ 
1804 6.37662000000008166722\\ 
2146 6.83689699999999156432\\ 
2556 8.10208500000020137577\\ 
};
\addlegendentry{$T_{\rm sort}$}

\end{loglogaxis}
\end{tikzpicture}%

}}%
\caption{\sf Left: convergence order study for Example 3. Data points are shown in small
  circles, while the dashed line is the least squares fitting of first $19$ data points,
  leading to an estimated convergence order $9.6$. Right: timing results for Example 3.
}
\label{example3b}
\end{figure}

\begin{figure}[!ht]
\centering 
\includegraphics[height=50mm]{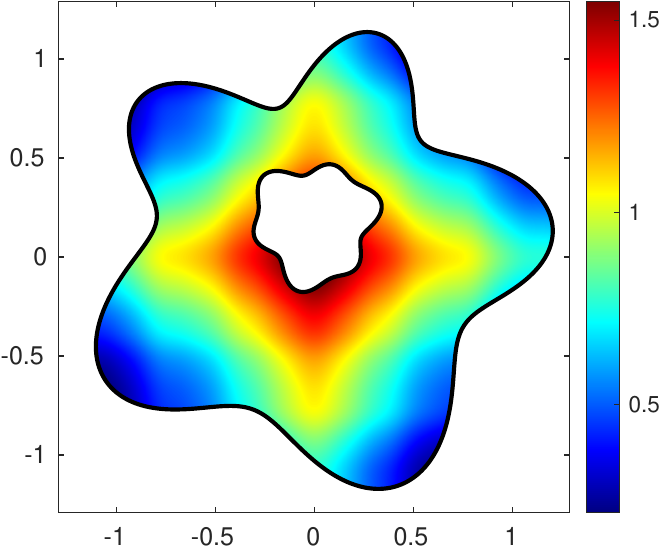}
\hspace{4mm}
\includegraphics[height=50mm]{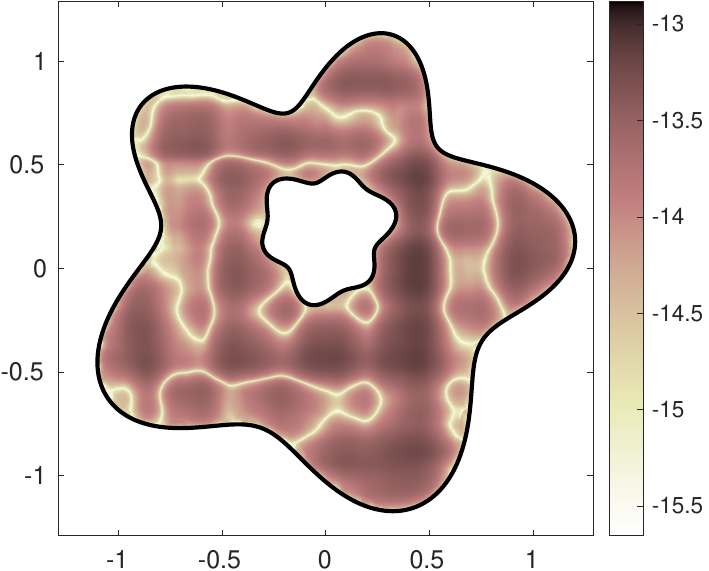}
\caption{\sf Left: solution to the Dirichlet Poisson problem in Example 3.
  Right: $\log_{10}$ of absolute error in the solution in $\Omega$. Here, the total number
  of uniform grid points is $2050^2$. The relative $l^2$ error is about $2\times 10^{-14}$
  on the grid.
}
\label{example3c}
\end{figure}

We set the spacing $h$ of the uniform grid to $25$ logarithmically
equally spaced points between $10^{-1.1}$ and $10^{-3}$. The number of points $N$
along each dimension varies from $128$ to $2556$, correspondingly. The convergence order and timing
results are shown in \cref{example3b}. We observe that the convergence order is
about $9.6$, and that the timing results are similar to the first two examples.
The solution to the Dirichlet Poisson problem and its numerical error 
are shown in \cref{example3c}.

Finally, we study the effect of function extension order on the accuracy of the solution.
Besides the first three examples, we also consider the case where
the boundary is a circle of radius $0.9$. The exact solution to the Dirichlet
Poisson problem is
\be
u_4(x) = T_{20}(x_1)+T_{20}(x_2),
\ee
the inhomogeneous term in the Poisson equation $f_4(x) = \Delta u_4$, and the Dirichlet data $g_4$
is simply the restriction of $u_4$ on the boundary. For each example we fix a grid spacing $h,$
and vary the extension order $K$ from $-1$ to $8$, where $K=-1$ denotes the
trivial extension by zero. \Cref{extensionorderstudy}
shows the dependence of accuracy of the solution on the function extension order for all
four examples. In all four cases, the accuracy increases
as the order of function extension increases.

\begin{figure}[!ht]
\centering 
\subfloat{{
%
%
\definecolor{mycolor1}{rgb}{0.00000,0.44700,0.74100}%
\definecolor{mycolor2}{rgb}{0.85000,0.32500,0.09800}%
\definecolor{mycolor3}{rgb}{0.92900,0.69400,0.12500}%
\definecolor{mycolor4}{rgb}{0.49400,0.18400,0.55600}%
\definecolor{mycolor5}{rgb}{0.46600,0.67400,0.18800}%
\definecolor{mycolor6}{rgb}{0.00000,0.74902,0.74902}%
\begin{tikzpicture}[scale=0.7]

\begin{semilogyaxis}[%
xticklabel style={/pgf/number format/fixed},
xmax=9,
xlabel style={font=\color{white!15!black}},
xlabel={$K$},
xticklabel style = {font=\small},
xminorticks=false,
ymin=1e-11,
ymax=1e-3,
yminorticks=true,
ylabel style={font=\small\color{white!15!black}},
ylabel={Relative $\ell^{2}$ error},
yticklabel style = {font=\small},
axis background/.style={fill=white},
xmajorgrids,
ymajorgrids,
legend style={font=\normalsize,at={(1.0,0.195)}, anchor=north east, legend cell align=left, align=left, draw=white!15!black},
clip mode=individual,
]

\addplot [color=mycolor1, line width=1.0pt, mark=*, mark options={solid, mycolor1}, mark size=2.0pt]
  table[row sep=crcr]{%
-1 0.00030208341203691335\\ 
0 0.00002109237718523421\\ 
1 0.00000199819572202127\\ 
2 0.00000038874414479117\\ 
3 0.00000003829339749951\\ 
4 0.00000000701182043846\\ 
5 0.00000000067330888856\\ 
6 0.00000000015080151088\\ 
7 0.00000000002515854597\\ 
8 0.00000000004200165157\\ 
};

\end{semilogyaxis}
\end{tikzpicture}%

}}%
\hspace{10mm}
\subfloat{{
%
%
\definecolor{mycolor1}{rgb}{0.00000,0.44700,0.74100}%
\definecolor{mycolor2}{rgb}{0.85000,0.32500,0.09800}%
\definecolor{mycolor3}{rgb}{0.92900,0.69400,0.12500}%
\definecolor{mycolor4}{rgb}{0.49400,0.18400,0.55600}%
\definecolor{mycolor5}{rgb}{0.46600,0.67400,0.18800}%
\definecolor{mycolor6}{rgb}{0.00000,0.74902,0.74902}%
\begin{tikzpicture}[scale=0.7]

\begin{semilogyaxis}[%
xticklabel style={/pgf/number format/fixed},
xmax=9,
xlabel style={font=\color{white!15!black}},
xlabel={$K$},
xticklabel style = {font=\small},
xminorticks=false,
ymin=1e-12,
ymax=1e-6,
yminorticks=false,
ylabel style={font=\small\color{white!15!black}},
ylabel={Relative $\ell^{2}$ error},
yticklabel style = {font=\small},
axis background/.style={fill=white},
xmajorgrids,
ymajorgrids,
legend style={font=\normalsize,at={(1.0,0.195)}, anchor=north east, legend cell align=left, align=left, draw=white!15!black},
clip mode=individual,
]

\addplot [color=mycolor1, line width=1.0pt, mark=*, mark options={solid, mycolor1}, mark size=2.0pt]
  table[row sep=crcr]{%
-1 0.00000016295575055718\\ 
0 0.00000000081832826668\\ 
1 0.00000000005270590071\\ 
2 0.00000000001876782391\\ 
3 0.00000000001202209894\\ 
4 0.00000000000747448175\\ 
5 0.00000000000683335784\\ 
6 0.00000000000694582578\\ 
7 0.00000000000699883697\\ 
8 0.00000000000699779443\\ 
};

\end{semilogyaxis}
\end{tikzpicture}%

}}%
\vspace{4mm}
\subfloat{{
%
%
\definecolor{mycolor1}{rgb}{0.00000,0.44700,0.74100}%
\definecolor{mycolor2}{rgb}{0.85000,0.32500,0.09800}%
\definecolor{mycolor3}{rgb}{0.92900,0.69400,0.12500}%
\definecolor{mycolor4}{rgb}{0.49400,0.18400,0.55600}%
\definecolor{mycolor5}{rgb}{0.46600,0.67400,0.18800}%
\definecolor{mycolor6}{rgb}{0.00000,0.74902,0.74902}%
\begin{tikzpicture}[scale=0.7]

\begin{semilogyaxis}[%
xticklabel style={/pgf/number format/fixed},
xmax=9,
xlabel style={font=\color{white!15!black}},
xlabel={$K$},
xticklabel style = {font=\small},
xminorticks=false,
ymin=1e-14,
ymax=1e-5,
yminorticks=true,
ylabel style={font=\small\color{white!15!black}},
ylabel={Relative $\ell^{2}$ error},
yticklabel style = {font=\small},
axis background/.style={fill=white},
xmajorgrids,
ymajorgrids,
legend style={font=\normalsize,at={(1.0,0.195)}, anchor=north east, legend cell align=left, align=left, draw=white!15!black},
clip mode=individual,
]

\addplot [color=mycolor1, line width=1.0pt, mark=*, mark options={solid, mycolor1}, mark size=2.0pt]
  table[row sep=crcr]{%
-1 0.00000085165153501279\\ 
0 0.00000000939527756316\\ 
1 0.00000000013652434462\\ 
2 0.00000000000125071868\\ 
3 0.00000000000003266028\\ 
4 0.00000000000001962081\\ 
5 0.00000000000002002357\\ 
6 0.00000000000002001914\\ 
7 0.00000000000002005811\\ 
8 0.00000000000001986801\\ 
};

\end{semilogyaxis}
\end{tikzpicture}%

}}%
\hspace{10mm}
\subfloat{{
%
%
\definecolor{mycolor1}{rgb}{0.00000,0.44700,0.74100}%
\definecolor{mycolor2}{rgb}{0.85000,0.32500,0.09800}%
\definecolor{mycolor3}{rgb}{0.92900,0.69400,0.12500}%
\definecolor{mycolor4}{rgb}{0.49400,0.18400,0.55600}%
\definecolor{mycolor5}{rgb}{0.46600,0.67400,0.18800}%
\definecolor{mycolor6}{rgb}{0.00000,0.74902,0.74902}%
\begin{tikzpicture}[scale=0.7]

\begin{semilogyaxis}[%
xticklabel style={/pgf/number format/fixed},
xmax=9,
xlabel style={font=\color{white!15!black}},
xlabel={$K$},
xticklabel style = {font=\small},
xminorticks=false,
ymin=1e-18,
ymax=1e-7,
yminorticks=true,
ylabel style={font=\small\color{white!15!black}},
ylabel={Relative $\ell^{2}$ error},
yticklabel style = {font=\small},
axis background/.style={fill=white},
xmajorgrids,
ymajorgrids,
legend style={font=\normalsize,at={(1.0,0.195)}, anchor=north east, legend cell align=left, align=left, draw=white!15!black},
clip mode=individual,
]

\addplot [color=mycolor1, line width=1.0pt, mark=*, mark options={solid, mycolor1}, mark size=2.0pt]
  table[row sep=crcr]{%
-1 0.00000000828929362635\\ 
0 0.00000000060830930515\\ 
1 0.00000000004019686342\\ 
2 0.00000000000919761899\\ 
3 0.00000000000020143128\\ 
4 0.00000000000009780492\\ 
5 0.00000000000000810587\\ 
6 0.00000000000000124456\\ 
7 0.00000000000000003592\\ 
8 0.00000000000000000989\\ 
};

\end{semilogyaxis}
\end{tikzpicture}%

}}%

\caption{\sf The dependence of accuracy of the solution on the function extension order.
  The $x$-axis is the order of function extension, where $K=-1$ denotes the
  trivial extension by zero. 
  The $y$-axis shows the relative $l^2$
  error of the numerical solution. Top left: the results for Example 1, where
  the numerical solution is calculated using a $334\times 334$ uniform grid with $h\approx 0.0076$.
  Top right: the results for Example 2, where the numerical solution
  is calculated using a $406\times 406$ uniform grid with $h\approx 0.0019$.
  Bottom left: the results for Example 3, where the numerical solution
  is calculated using a $784\times 784$ uniform grid with $h\approx 0.0036$.
  Bottom right: the results for Example 4, where the numerical solution
  is calculated using a $402\times 402$ uniform grid with $h\approx 0.0059$.
}
\label{extensionorderstudy}
\end{figure}
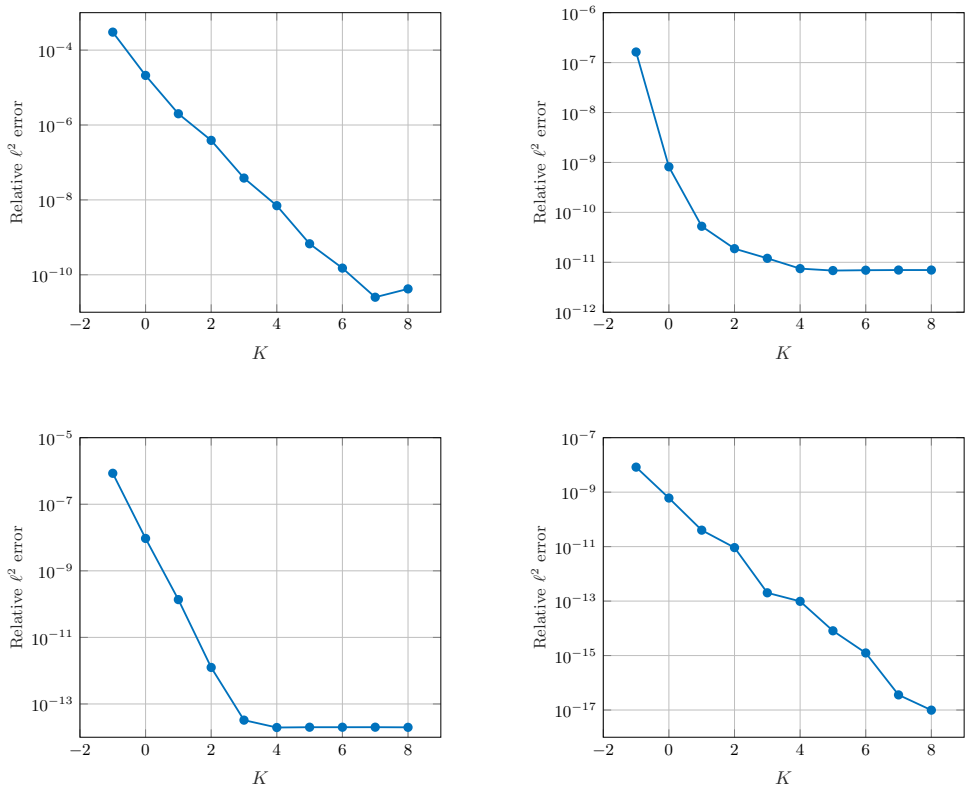

\subsection{Three dimensions -- FFT-based BVP solvers}\label{3dexamples}
We consider two examples in three dimensions to solve the inhomogeneous Dirichlet problem
for the Poisson equation using an FFT-based BVP solver. As in the two-dimensional case
presented in the previous subsection, function extension is combined with the FFT
to generate a particular solution $u^{p}$. We present timings and convergence plots for
the errors with respect to the grid resolution $h$ of the underlying uniform grid.
Again, we set the function extension order to $8$. For the first example we set the
parameter $a=1$ in the formulae \cref{optnodes}--\cref{optweights} for computing
extension nodes and weights, while in the second example we set $a=0.15$. The parameters
in the window function \cref{windowfunc} are set to $c=40.590000152587891$ for $16$ digits
of accuracy, and $r_0=10^{-6}$, $r_1=32h$. The error is measured at the subset of
$10000$ random points that are inside the given domain, $\Omega$. These random points are
drawn from a uniform distribution over a box with sides such that it just
contains $\Omega$. For the first example the box is $[-0.25,0.25]^{3}$,
and for the second example the box is $[-0.19,0.19]\times[-0.18,0.18]\times[-0.062,0,62]$.

In the first example the inhomogeneous term in the Poisson equation is given by
\begin{equation}
  \label{eq:3drhs1}
  f_{4}(x) = -\sin(8\pi x_{1})\sin(8\pi x_{2})\sin(8\pi x_{3}), \quad x\in\Omega_{4},
\end{equation}
with boundary data
\begin{equation}
  \label{eq:3dbc1}
  g_{4}(x) = \sin(8\pi x_{1})\sin(8\pi x_{2})\sin(8\pi x_{3})/(1536\pi^{3}), \quad x\in\Gamma_{4}.
\end{equation}
Here, $\Omega_{4}$ is the ball with radius $0.25$. The exact solution $u_{4}(x)$ is given by $g_{4}(x)$ for $x\in\Omega_{4}$.

The grid spacing $h$ of the uniform grid ranges over $23$ logarithmically equally
distributed points between $10^{-2}$ and $10^{-2.75}$, corresponding to $N = 120$ and $N=358$, respectively.
In \cref{fig:3dexample4convtimings} we observe the expected $10$th order convergence in the error
as a function of the grid resolution $h$. The error saturates $10^{-9}$ for values of $h$ less
than $10^{-2.6}$, which is probably due to package used for solving the homogeneous
Laplace equation \cite{GREENGARD2021100092}. An extension $f^{e}_{4}$ and the Dirichlet data
are shown in \cref{fig:3dexample4fesol}.

In \cref{fig:3dexample4convtimings} timings are presented. For the three-dimensional setting
there are two dominant steps: solving the Laplace equation and sorting points in the
uniform grid. As is clear from the plots, function extension and the FFT require relatively little work.

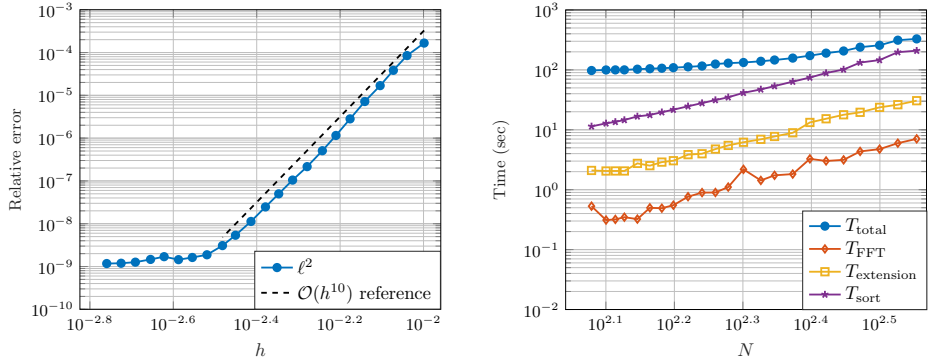
\begin{figure}[!ht]
\centering 
\subfloat{{
%
%
\definecolor{mycolor1}{rgb}{0.00000,0.44700,0.74100}%
\definecolor{mycolor2}{rgb}{0.85000,0.32500,0.09800}%
\definecolor{mycolor3}{rgb}{0.92900,0.69400,0.12500}%
\definecolor{mycolor4}{rgb}{0.49400,0.18400,0.55600}%
\definecolor{mycolor5}{rgb}{0.46600,0.67400,0.18800}%
\definecolor{mycolor6}{rgb}{0.00000,0.74902,0.74902}%
\begin{tikzpicture}[scale=0.7]

\begin{loglogaxis}[%
xticklabel style={/pgf/number format/fixed},
xmin=0.0015,
xmax=0.011,
xlabel style={font=\color{white!15!black}},
xlabel={$h$},
xticklabel style = {font=\small},
xminorticks=false,
ymin=1e-10,
ymax=1e-3,
yminorticks=true,
ylabel style={font=\small\color{white!15!black}},
ylabel={Relative error},
yticklabel style = {font=\small},
axis background/.style={fill=white},
xmajorgrids,
xminorgrids,
ymajorgrids,
yminorgrids,
legend style={font=\normalsize,at={(1.0,0.195)}, anchor=north east, legend cell align=left, align=left, draw=white!15!black},
clip mode=individual,
]

\addplot [color=mycolor1, line width=1.0pt, mark=*, mark options={solid, mycolor1}, mark size=2.0pt]
  table[row sep=crcr]{%
   0.010033333333333   0.000166496045466\\
   0.009129081614761   0.000084993868957\\
   0.008466385908475   0.000038390617185\\
   0.007871526733009   0.000016862323170\\
   0.007231702504403   0.000007171470016\\
   0.006666605277874   0.000002818228736\\
   0.006165762678235   0.000001153811474\\
   0.005720412257820   0.000000508149146\\
   0.005259045242183   0.000000215633513\\
   0.004853624279660   0.000000104713940\\
   0.004495785150592   0.000000049931264\\
   0.004178644490178   0.000000024706849\\
   0.003857530388250   0.000000011293368\\
   0.003541425714502   0.000000005379840\\
   0.003295676834175   0.000000003082599\\
   0.003025251503876   0.000000001880654\\
   0.002790512083030   0.000000001625426\\
   0.002585391073247   0.000000001446594\\
   0.002387892407311   0.000000001698568\\
   0.002215337928687   0.000000001464651\\
   0.002037507987776   0.000000001263322\\
   0.001883548762733   0.000000001192489\\
   0.001739470248835   0.000000001173061\\
};
\addlegendentry{$\ell^{2}$}

\addplot [color=black, dashed, line width=1.0pt, mark options={solid, black}]
  table[row sep=crcr]{%
   0.010033333333333   0.000327838813674\\
   0.009129081614761   0.000127490803436\\
   0.008466385908475   0.000060005155223\\
   0.007871526733009   0.000028959943352\\
   0.007231702504403   0.000012405477357\\
   0.006666605277874   0.000005498641162\\
   0.006165762678235   0.000002518117591\\
   0.005720412257820   0.000001189819950\\
   0.005259045242183   0.000000513188635\\
   0.004853624279660   0.000000230075039\\
   0.004495785150592   0.000000106970357\\
   0.004178644490178   0.000000051470952\\
   0.003857530388250   0.000000023136706\\
   0.003541425714502   0.000000009839869\\
   0.003295676834175   0.000000004793516\\
};
\addlegendentry{$\mathcal{O}(h^{10})$ reference}

\end{loglogaxis}
\end{tikzpicture}%

}}%
\hspace{4mm}
\subfloat{{
%
%
\definecolor{mycolor1}{rgb}{0.00000,0.44700,0.74100}%
\definecolor{mycolor2}{rgb}{0.85000,0.32500,0.09800}%
\definecolor{mycolor3}{rgb}{0.92900,0.69400,0.12500}%
\definecolor{mycolor4}{rgb}{0.49400,0.18400,0.55600}%
\definecolor{mycolor5}{rgb}{0.46600,0.67400,0.18800}%
\definecolor{mycolor6}{rgb}{0.00000,0.74902,0.74902}%
\begin{tikzpicture}[scale=0.7]

\begin{loglogaxis}[%
xticklabel style={/pgf/number format/fixed},
xmin=110,
xmax=370,
xlabel style={font=\color{white!15!black}},
xlabel={$N$},
xticklabel style = {font=\small},
xminorticks=false,
ymin=1e-2,
ymax=1e3,
yminorticks=true,
ylabel style={font=\small\color{white!15!black}},
ylabel={Time (sec)},
yticklabel style = {font=\small},
axis background/.style={fill=white},
xmajorgrids,
xminorgrids,
ymajorgrids,
yminorgrids,
legend style={font=\normalsize,at={(1.002,0.33)}, anchor=north east, legend cell align=left, align=left, draw=white!15!black},
clip mode=individual,
]

\addplot [color=mycolor1, line width=1.0pt, mark=*, mark options={solid, mycolor1}, mark size=2.0pt]
  table[row sep=crcr]{%
120 97.502541\\ 
126 99.611096\\ 
130 99.880003\\ 
134 99.918898\\ 
140 102.632604\\ 
146 104.516630\\ 
152 106.469955\\ 
158 108.741874\\ 
166 112.280464\\ 
174 115.954080\\ 
182 124.962273\\ 
190 129.384871\\ 
200 132.657644\\ 
212 139.070138\\ 
222 146.023395\\ 
236 157.193427\\ 
250 172.851646\\ 
264 190.199742\\ 
280 206.548223\\ 
296 239.249006\\ 
316 257.207608\\ 
336 312.700081\\ 
358 328.860173\\ 
};
\addlegendentry{$T_{\rm total}$}

\addplot [color=mycolor2, line width=1.0pt, mark=diamond, mark options={solid, mycolor2}, mark size=2.0pt]
  table[row sep=crcr]{%
120 0.530833\\ 
126 0.311074\\ 
130 0.319286\\ 
134 0.347307\\ 
140 0.322813\\ 
146 0.498237\\ 
152 0.490606\\ 
158 0.550588\\ 
166 0.764516\\ 
174 0.893820\\ 
182 0.897672\\ 
190 1.106392\\ 
200 2.184867\\ 
212 1.426934\\ 
222 1.736638\\ 
236 1.822811\\ 
250 3.270351\\ 
264 3.016378\\ 
280 3.145589\\ 
296 4.350317\\ 
316 4.752137\\ 
336 5.973768\\ 
358 7.053649\\ 
};
\addlegendentry{$T_{\rm FFT}$}

\addplot [color=mycolor3, line width=1.0pt, mark=square, mark options={solid, mycolor3}, mark size=2.0pt]
  table[row sep=crcr]{%
120 2.094437\\ 
126 2.051101\\ 
130 2.053938\\ 
134 2.052798\\ 
140 2.747790\\ 
146 2.517480\\ 
152 2.886046\\ 
158 3.045066\\ 
166 3.845621\\ 
174 3.983258\\ 
182 4.762392\\ 
190 5.480643\\ 
200 6.190172\\ 
212 6.927492\\ 
222 7.732064\\ 
236 8.904037\\ 
250 13.252774\\ 
264 15.269597\\ 
280 17.802219\\ 
296 19.566861\\ 
316 23.669149\\ 
336 26.204722\\ 
358 30.574936\\
};
\addlegendentry{$T_{\rm extension}$}

\addplot [color=mycolor4, line width=1.0pt, mark=star, mark options={solid, mycolor4}, mark size=2.0pt]
  table[row sep=crcr]{%
120 11.389299\\ 
126 12.793415\\ 
130 13.542138\\ 
134 14.555549\\ 
140 16.718072\\ 
146 17.664844\\ 
152 19.633690\\ 
158 21.622125\\ 
166 24.681743\\ 
174 27.872978\\ 
182 31.330231\\ 
190 34.842309\\ 
200 41.381263\\ 
212 47.140944\\ 
222 53.576720\\ 
236 63.586855\\ 
250 74.423585\\ 
264 88.569126\\ 
280 101.334956\\ 
296 132.336051\\ 
316 145.752560\\ 
336 197.405767\\ 
358 208.874156\\ 
};
\addlegendentry{$T_{\rm sort}$}

\end{loglogaxis}
\end{tikzpicture}%

}}%
\caption{\sf Left: convergence order study for Example 4. Data points are shown in small
  circles, while the dashed line is a reference line corresponding to $\mathcal{O}(h^{10})$.
  Right: timing results for Example 4, where the x-axis is the number of points along each
  dimension. Thus, the total number of equispaced grid points in the volume is $N^3$.
}
\label{fig:3dexample4convtimings}
\end{figure}

\begin{figure}[!ht]
\centering 
\includegraphics[trim={0cm 5cm 1cm 5.5cm},clip,height=50mm]{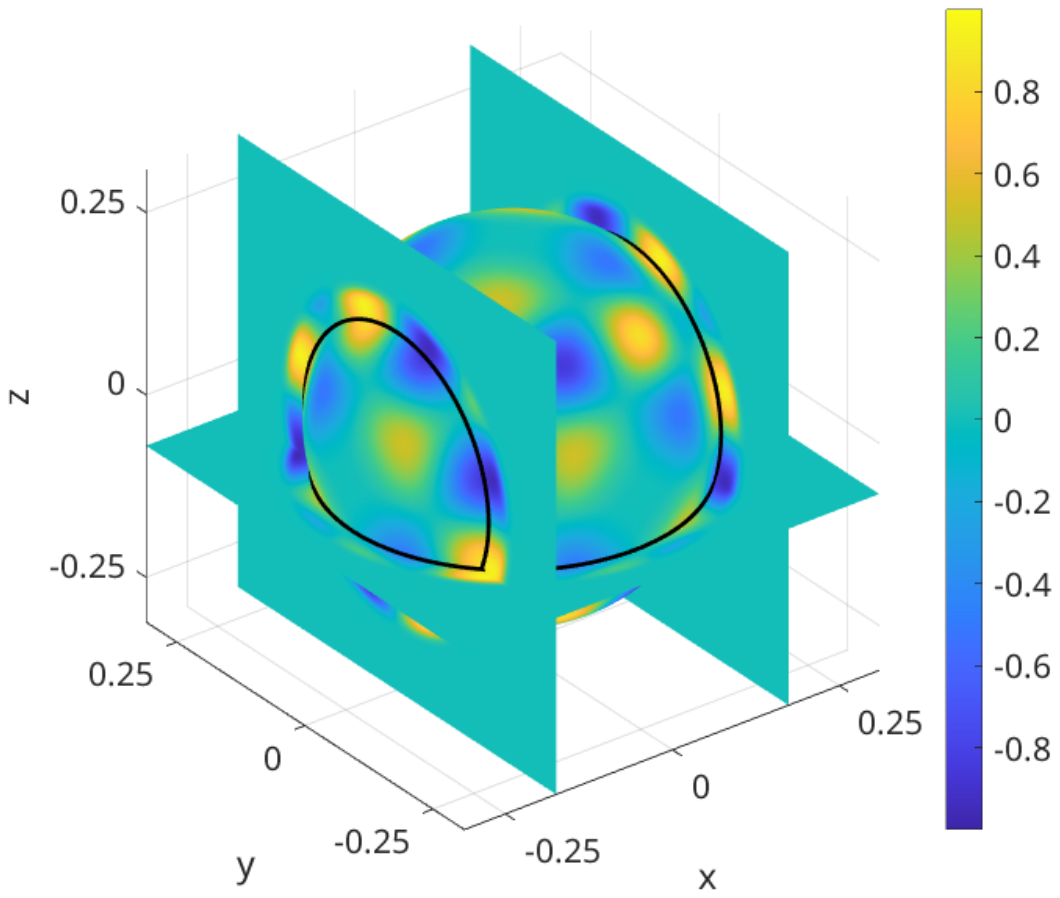}
\hspace{4mm}
\includegraphics[trim={0cm 5cm 1cm 5.5cm},clip,height=50mm]{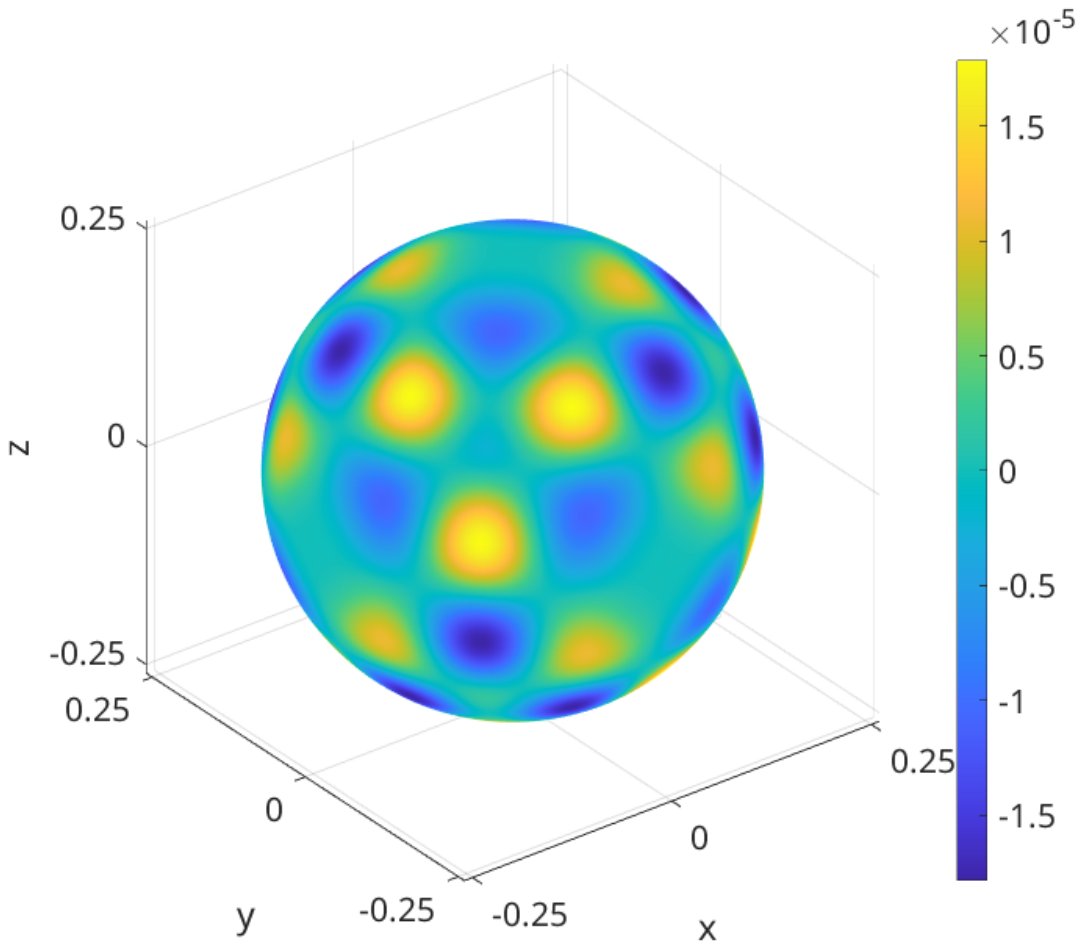}
\caption{\sf Left: the surface $\Gamma_{4}$, the right-hand side $f_{4}$ on the surface,
  and slices of its 8th order normal extension $f_{4}^{e}$ with $h = 10^{-2.75}$.
  Right: the Dirichlet boundary data $g_{4}$ on the surface $\Gamma_{4}$.
}
\label{fig:3dexample4fesol}
\end{figure}

For the second example we consider a stellarator-like geometry, given by a
surface $\Gamma_{5}$  parameterized by $\mathbf{X}:[0,2\pi]^{2}\rightarrow\Gamma_{5}$ with
\begin{equation}
  \mathbf{X}(u,v) = \sum\limits_{i=-1}^{2}\sum\limits_{j=-1}^{1}\delta_{i,j}
  \begin{bmatrix}\cos(v)\cos((1-i)u+jv)\\\sin(v)\cos((1-i)u+jv)\\\sin((1-i)u+jv)\end{bmatrix}.
\end{equation}
The non-zero coefficients are $\delta_{-1,-1}=0.17$, $\delta_{-1,0}=0.11$, $\delta_{0,0}=1$,
$\delta_{1,0}=4.5$, $\delta_{2,0}=-0.25$, $\delta_{0.1}=0.07$, and $\delta_{2,1}=-0.45$.
We let $\Omega_5$ denote the solid torus bounded by $\Gamma_5,$ see \cref{fig:3dexample2fesol}.

We use the following inhomogeneous term for the Poisson equation,
\begin{align}
  \begin{split}
  \label{eq:3drhs2}
  f_{5}(x) &= -300\sin(10(x_{1} + x_{2} + x_{3})) + 2 + 6x_{3} \\
  &+ 1000e^{-500x^{2}}(500x_{1}^{2} -1)+ 1000e^{-500x_{3}^{2}}(500x_{3}^{2}-1)
    , \quad x\in\Omega_{5},
    \end{split}
\end{align}
with boundary data
\begin{equation}
  \label{eq:3dbc2}
  g_{5}(x) = \sin(10(x_{1} + x_{2} + x_{3})) + x_{1}^{2} - 3x_{2} + x_{3}^{3} + e^{-500x_{1}^{2}}+e^{-500x_{3}^{2}},\quad x\in\Gamma_{5}.
\end{equation}
Again, the exact solution $u_{5}(x)$  is given by $g_{5}(x)$ for $x\in\Omega_{5}$.
The plots in \cref{fig:3dexample2fesol} show the stellarator-like geometry,
an extension $f_{5}^{e}$ and the Dirichlet boundary data $g_{5}$.

For this example, the grid spacing $h$ of the uniform grid ranges over $22$ logarithmically
equally distributed points between $10^{-2.2}$ to $10^{-2.9}$, which corresponds to $N=126$ and $N=364$,
respectively. The error as a function of $h$ is shown in \cref{fig:3dexample2convtimings}. As for
the previous example, we observe the expected 10th order convergence; the error is saturated
for values of $h$ less than $10^{-2.62}$. As remarked above, this saturation is probably due
to limitations in the solver for the homogeneous Laplace equation.

Due to the presence of very narrow regions in  $\Omega_{5}$, we have to set $a=0.15$, or smaller,
to ensure that the points $y-t_{j}x\nu_{y}$ (see formula \cref{extformula}) for each boundary
point $y$ all lie inside $\Omega_{5}$. Note that $a=0.15$ gives $\|w\|_{1}\sim \mathcal{O}(10^{11})$,
however it is clear from the convergence plot in \cref{fig:3dexample2convtimings} that
we nonetheless get the same accuracy as in the previous example, which uses the sphere with $a=1$.

Concerning the timings, we observe the same trends as for the previous example
in three dimensions: solving the Laplace equation and sorting the points are the
two most time consuming tasks.
\begin{figure}[!ht]
\centering 
\subfloat{{
%
%
\definecolor{mycolor1}{rgb}{0.00000,0.44700,0.74100}%
\definecolor{mycolor2}{rgb}{0.85000,0.32500,0.09800}%
\definecolor{mycolor3}{rgb}{0.92900,0.69400,0.12500}%
\definecolor{mycolor4}{rgb}{0.49400,0.18400,0.55600}%
\definecolor{mycolor5}{rgb}{0.46600,0.67400,0.18800}%
\definecolor{mycolor6}{rgb}{0.00000,0.74902,0.74902}%
\begin{tikzpicture}[scale=0.7]

\begin{loglogaxis}[%
xticklabel style={/pgf/number format/fixed},
xmin=0.0012,
xmax=0.007,
xlabel style={font=\color{white!15!black}},
xlabel={$h$},
xticklabel style = {font=\small},
xminorticks=false,
ymin=1e-10,
ymax=1e-3,
yminorticks=true,
ylabel style={font=\small\color{white!15!black}},
ylabel={Relative error},
yticklabel style = {font=\small},
axis background/.style={fill=white},
xmajorgrids,
xminorgrids,
ymajorgrids,
yminorgrids,
legend style={font=\normalsize,at={(0.515,1.001)}, anchor=north east, legend cell align=left, align=left, draw=white!15!black},
clip mode=individual,
]

\addplot [color=mycolor1, line width=1.0pt, mark=*, mark options={solid, mycolor1}, mark size=2.0pt]
  table[row sep=crcr]{%
   0.006703618351288   0.000066225887747\\
   0.006219442067573   0.000030332765925\\
   0.005699681459032   0.000008828817098\\
   0.005316664656394   0.000001897101761\\
   0.004903165522970   0.000000493363988\\
   0.004477691162991   0.000000206980989\\
   0.004157962467667   0.000000090257636\\
   0.003826414400849   0.000000036106906\\
   0.003534737461178   0.000000017116927\\
   0.003277020646085   0.000000007777758\\
   0.003016963384575   0.000000004709844\\
   0.002789017462991   0.000000003866927\\
   0.002588199930047   0.000000003601918\\
   0.002389121308584   0.000000003603761\\
   0.002214586422622   0.000000003520174\\
   0.002044381213611   0.000000004155074\\
   0.001881357932728   0.000000004167294\\
   0.001739654544985   0.000000004435131\\
   0.001604946981566   0.000000004723448\\
   0.001487628531395   0.000000004977825\\
   0.001368563558254   0.000000005381951\\
};
\addlegendentry{$\ell^{2}$}

\addplot [color=black, dashed, line width=1.0pt, mark options={solid, black}]
  table[row sep=crcr]{%
   0.005720412257820   0.000001189819950\\
   0.005259045242183   0.000000513188635\\
   0.004853624279660   0.000000230075039\\
   0.004495785150592   0.000000106970357\\
   0.004178644490178   0.000000051470952\\
   0.003857530388250   0.000000023136706\\
   0.003541425714502   0.000000009839869\\
   0.003295676834175   0.000000004793516\\
};
\addlegendentry{$\mathcal{O}(h^{10})$ reference}

\end{loglogaxis}
\end{tikzpicture}%

}}%
\hspace{4mm}
\subfloat{{
%
%
\definecolor{mycolor1}{rgb}{0.00000,0.44700,0.74100}%
\definecolor{mycolor2}{rgb}{0.85000,0.32500,0.09800}%
\definecolor{mycolor3}{rgb}{0.92900,0.69400,0.12500}%
\definecolor{mycolor4}{rgb}{0.49400,0.18400,0.55600}%
\definecolor{mycolor5}{rgb}{0.46600,0.67400,0.18800}%
\definecolor{mycolor6}{rgb}{0.00000,0.74902,0.74902}%
\begin{tikzpicture}[scale=0.7]

\begin{loglogaxis}[%
xticklabel style={/pgf/number format/fixed},
xmin=110,
xmax=370,
xlabel style={font=\color{white!15!black}},
xlabel={$N$},
xticklabel style = {font=\small},
xminorticks=false,
ymin=1e-2,
ymax=1e3,
yminorticks=true,
ylabel style={font=\small\color{white!15!black}},
ylabel={Time (sec)},
yticklabel style = {font=\small},
axis background/.style={fill=white},
xmajorgrids,
xminorgrids,
ymajorgrids,
yminorgrids,
legend style={font=\normalsize,at={(1.002,0.33)}, anchor=north east, legend cell align=left, align=left, draw=white!15!black},
clip mode=individual,
]

\addplot [color=mycolor1, line width=1.0pt, mark=*, mark options={solid, mycolor1}, mark size=2.0pt]
  table[row sep=crcr]{%
126 89.665939\\ 
130 92.762465\\ 
136 94.403573\\ 
140 95.590694\\ 
146 96.649697\\ 
154 100.558252\\ 
160 101.085612\\ 
168 105.622958\\ 
176 110.326599\\ 
184 111.877881\\ 
194 119.108153\\ 
204 122.598084\\ 
214 129.420581\\ 
226 136.735181\\ 
238 141.881145\\ 
252 155.535792\\ 
268 167.129662\\ 
284 180.336998\\ 
302 203.902445\\ 
320 222.943451\\ 
342 249.375298\\ 
};
\addlegendentry{$T_{\rm total}$}

\addplot [color=mycolor2, line width=1.0pt, mark=diamond, mark options={solid, mycolor2}, mark size=2.0pt]
  table[row sep=crcr]{%
126 0.272788\\ 
130 0.284386\\ 
136 0.405909\\ 
140 0.368241\\ 
146 0.485085\\ 
154 0.509074\\ 
160 0.636283\\ 
168 0.708914\\ 
176 1.016748\\ 
184 0.955126\\ 
194 2.151137\\ 
204 1.340866\\ 
214 1.629706\\ 
226 1.909775\\ 
238 1.494953\\ 
252 2.415421\\ 
268 2.590744\\ 
284 3.480662\\ 
302 4.271853\\ 
320 4.709987\\ 
342 5.926126\\ 
};
\addlegendentry{$T_{\rm FFT}$}

\addplot [color=mycolor3, line width=1.0pt, mark=square, mark options={solid, mycolor3}, mark size=2.0pt]
  table[row sep=crcr]{%
126 2.356729\\ 
130 2.853321\\ 
136 2.855850\\ 
140 3.434882\\ 
146 3.429810\\ 
154 4.731237\\ 
160 4.433298\\ 
168 5.442161\\ 
176 5.830059\\ 
184 6.254142\\ 
194 7.387302\\ 
204 8.431283\\ 
214 9.836101\\ 
226 10.511290\\ 
238 11.431320\\ 
252 13.590632\\ 
268 16.200780\\ 
284 16.839493\\ 
302 21.776690\\ 
320 25.029531\\ 
342 28.153079\\ 
};
\addlegendentry{$T_{\rm extension}$}

\addplot [color=mycolor4, line width=1.0pt, mark=star, mark options={solid, mycolor4}, mark size=2.0pt]
  table[row sep=crcr]{%
126 7.880032\\ 
130 8.459279\\ 
136 9.395193\\ 
140 10.424591\\ 
146 11.769201\\ 
154 13.705887\\ 
160 14.596205\\ 
168 18.081616\\ 
176 20.318485\\ 
184 21.639825\\ 
194 25.584882\\ 
204 29.808292\\ 
214 34.567248\\ 
226 39.428883\\ 
238 45.870803\\ 
252 55.317842\\ 
268 67.090987\\ 
284 76.146730\\ 
302 93.971464\\ 
320 108.746587\\ 
342 132.381813\\ 
};
\addlegendentry{$T_{\rm sort}$}

\end{loglogaxis}
\end{tikzpicture}%

}}%
\caption{\sf Left: convergence order study for Example 5. Data points are shown in small
  circles, while the dashed line is a reference line corresponding to $\mathcal{O}(h^{10})$.
  Right: timing results for Example 5.
}
\label{fig:3dexample2convtimings}
\end{figure}
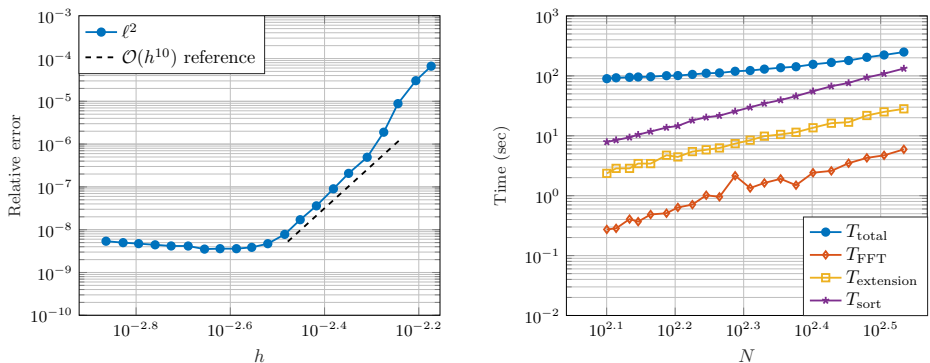

\begin{figure}[!ht]
\centering
\includegraphics[trim={0cm 7cm 1cm 6cm},clip,height=42mm]{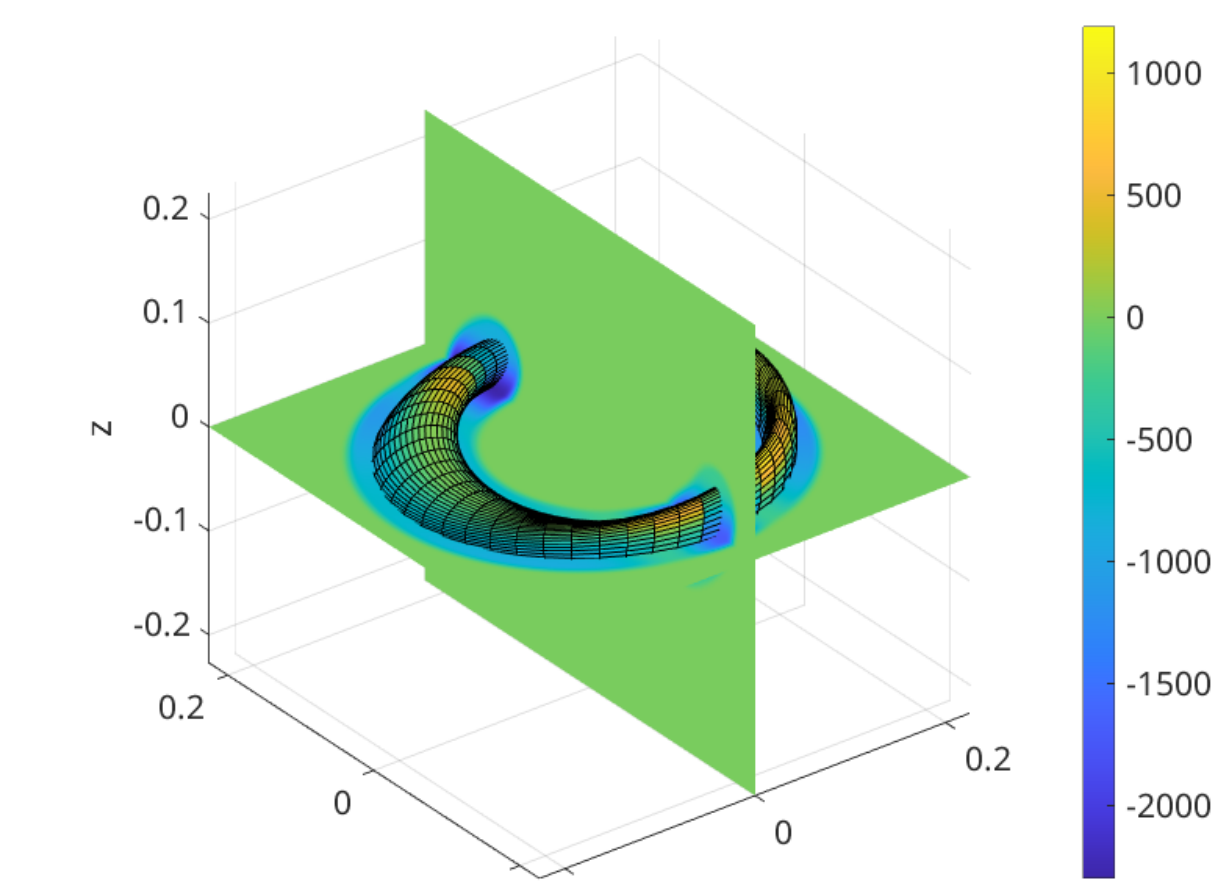}
\hspace{4mm}
\includegraphics[trim={0cm 8cm 1cm 6cm},clip,height=42mm]{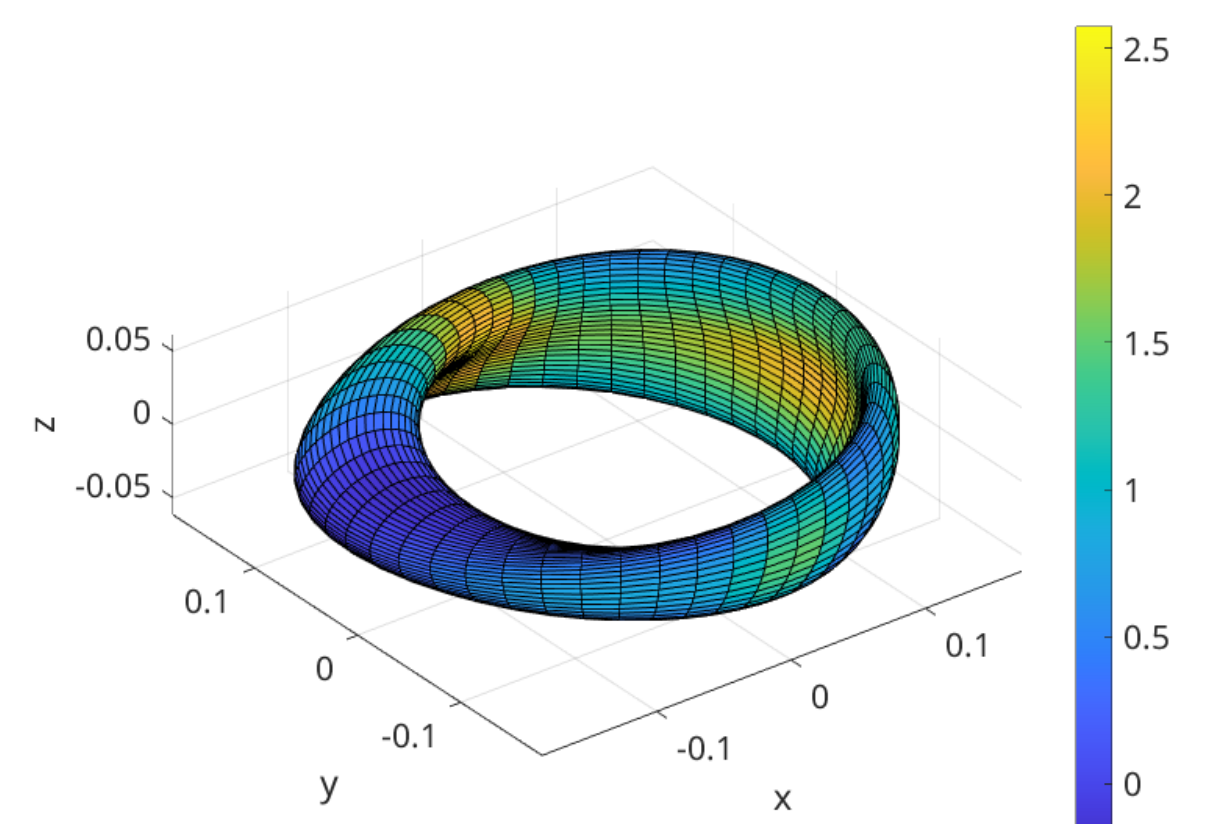}
\caption{\sf Left: the surface $\Gamma_{5}$, the right-hand side $f_{5}$ on the surface,
  and slices of its 8th order normal extension $f_{5}^{e}$ with $h = 10^{-2.9}$.
  Right: the Dirichlet boundary data $g_{5}$ on the surface $\Gamma_{5}$.
}
\label{fig:3dexample2fesol}
\end{figure}

\section{Conclusions and further discussions}
We have constructed a function extension scheme along normal directions. The extension
formula is given by a linear combination of function values in the specified domain,
with explicit expressions for nodes and weights that can be pre-computed. The nodes are
Chebyshev nodes scaled by the distance from the extension point to the boundary
in the normal direction. Each point in the $p$th order extension necessitates the evaluation of
the $p+1$ function values in the original domain. Moreover, the smoothness order of
the extended part along tangential directions remains consistent with the original
function. We have incorporated this new function extension scheme into an FFT-based
elliptic BVP solver on a uniform grid. The resulting scheme is both accurate and efficient.

This scheme can also be integrated with adaptive volume grids, such as the box FMM
(fast multipole method). Consequently, the modified scheme can tackle elliptic BVPs
with highly nonuniform data in complex geometries. In practical applications, we have
observed that the function extension scheme remains robust; the extended function
retains high accuracy even when the extension direction slightly deviates from the
normal direction. It is feasible to extend this new scheme along coordinate lines,
potentially resulting in more efficient function extension schemes when the function
in the original domain is provided, for instance, via polynomial interpolation on each element.
These issues are currently under investigation.

\section*{Acknowledgments}
The authors would like to thank Travis Askham at New Jersey Institute of Technology
for helpful discussions. The second author gratefully acknowledges the
support from the Knut and Alice Wallenberg Foundation under grant 2020.0258.

\bibliographystyle{siam}
\bibliography{journalnames,fmm}
\medskip

\hfill\today
\end{document}